\documentclass[a4paper,12pt]{amsart}
\usepackage{latexsym}
\usepackage{amsmath,amssymb}
\usepackage{amsmath}
\usepackage{hyperref}

\usepackage[all]{xy}
\usepackage{enumerate}
\usepackage[usenames]{color}
\usepackage{xcolor}

\newtheorem{theo}{Theorem}[section]
\newtheorem{coll}[theo]{Corollary}
\newtheorem{lemm}[theo]{Lemma}
\newtheorem{prop}[theo]{Proposition}
\newtheorem{defn}[theo]{Definition}
\newtheorem{ex}[theo]{Example}
\newtheorem{rem}[theo]{Remark}

\newcommand{\Hom}{{\rm Hom}}
\newcommand{\End}{{\rm End}}

\def\Im{\mbox{Im\/}}

\newcommand{\A}{\mathcal A}

\begin{document}
\sloppy

\title[CS-Baer and dual CS-Baer objects]{CS-Baer and dual CS-Baer objects \\ in abelian categories}

\author[S. Crivei]{Septimiu Crivei}
\address{Faculty of Mathematics and Computer Science, Babe\c s-Bolyai University, Str. M. Kog\u alniceanu 1,
400084 Cluj-Napoca, Romania} \email{crivei@math.ubbcluj.ro}

\author[D. Kesk{i}n T\"ut\"unc\"u]{Derya Kesk{i}n T\"ut\"unc\"u}
\address {Department of Mathematics, Hacettepe University, 06800 Beytepe, Ankara, Turkey} \email{keskin@hacettepe.edu.tr}

\author[S.M. Radu]{Simona Maria Radu}
\address{Faculty of Mathematics and Computer Science, Babe\c s-Bolyai University, Str. M. Kog\u alniceanu 1,
400084 Cluj-Napoca, Romania} \email{simonamariar@math.ubbcluj.ro}

\author[R. Tribak]{Rachid Tribak}
\address{Centre R\'{e}gional des M\'{e}tiers de l'Education et de la Formation (CRMEF-TTH)-Tanger, 
Avenue My Abdelaziz, Souani, B.P. 3117, Tangier, Morocco} \email{tribak12@yahoo.com}



\begin{abstract} We investigate relative CS-Baer objects in abelian categories in relationship with other relevant classes of objects such as relative Baer objects, extending objects, objects having certain summand intersection properties and relative CS-Rickart objects. Dual results are automatically obtained by applying the duality principle in abelian categories. We also study direct sums of relative CS-Baer objects, and we determine the complete structure of dual self-CS-Baer modules over Dedekind domains. Further applications are given to module categories.   
\end{abstract}

\date{January 24, 2022}

\maketitle

%\tableofcontents

\section{Introduction}

Let $M$ and $N$ be modules over a unitary ring. Then $N$ is called \emph{$M$-CS-Rickart} if the kernel of every homomorphism $f:M\to N$ is essential in a direct summand of $M$, and \emph{dual $M$-CS-Rickart} if the image of every homomorphism $f:M\to N$ lies above a direct summand of $N$ in the sense that ${\rm Im}(f)/L$ is superfluous in $M/L$ for some direct summand $L$ of $M$. Inspired by the work of Abyzov and Nhan on (dual) CS-Rickart modules \cite{AN1,ANQ}, we have considered and studied (dual) relative CS-Rickart objects in abelian categories \cite{CR1,CR2}, as a generalization of both extending (lifting) and (dual) relative Rickart objects \cite{CK}. The concepts of extending and lifting modules are well established in Module Theory \cite{CLVW,DHSW}. The interest in (dual) relative CS-Rickart objects stems, on one hand, from the module-theoretic work of Lee, Rizvi and Roman \cite{LRR10,LRR11}, and on the other hand, from their applications to relative regular objects in abelian categories studied by D\u asc\u alescu, N\u ast\u asescu, Tudorache and D\u au\c s \cite{DNTD}, as shown by Crivei and K\"or \cite{CK}. The setting of abelian categories has the advantage that one may freely use the duality principle in order to automatically obtain dual results, and yields applications to categories other than module categories. 
Furthermore, (dual) relative Rickart objects generalize (dual) relative Baer objects, which were studied by Rizvi and Roman \cite{RR04,RR09}, Keskin T\"ut\"unc\"u and Tribak \cite{KT} in module categories, and Crivei and K\"or \cite{CK} in abelian categories. Likewise, as a relevant subclass of (dual) relative CS-Rickart objects, in the present paper we will consider and investigate (dual) relative CS-Baer objects in abelian categories. Note that they were previously considered in module categories by Nhan \cite{Nhan}. The above notions have some strong versions, obtained by replacing direct summands by fully invariant direct summands in their definitions. E.g., see the work of Al-Saadi and Ibrahiem on (dual) strongly Rickart modules \cite{AI14,AI15}, Ebrahimi Atani, Khoramdel and Dolati Pish Hesari \cite{EKD} on strongly extending modules, Wang \cite{Wang} on strongly lifting modules, Crivei and Olteanu on (dual) strongly Rickart objects in abelian categories \cite{CO1,CO2}, Crivei and Radu on (dual) strongly CS-Rickart objects in abelian categories \cite{CR3}.  

The paper is structured as follows. We shall mainly refer to (strongly) relative CS-Baer objects in an abelian category $\mathcal{A}$ with AB3*. In Section 3 we introduce relative CS-Baer and strongly relative CS-Baer objects in abelian categories, and we illustrate them by a series of examples. We also show that an object $M$ is strongly self-CS-Baer if and only if $M$ is self-CS-Baer and weak duo if and only if $M$ is self-CS-Baer and ${\rm End}_{\A}(M)$ is abelian. 

Sections 4 and 5 are dedicated to the comparison of (strongly) relative CS-Baer objects with their most important particularizations, namely (strongly) relative Baer and (strongly) extending objects. For objects $M$ and $N$, we show that $N$ is (strongly) $M$-Baer if and only if $N$ is (strongly) $M$-CS-Baer and $M$-$\mathcal{K}$-nonsingular. In particular, every non-singular (strongly) self-CS-Baer right $R$-module is (strongly) self-Baer and every non-cosingular dual (strongly) self-CS-Baer right $R$-module is dual (strongly) self-Baer. We consider some stronger forms of $M$-$\mathcal{K}$-nonsingularity and $M$-$\mathcal{K}$-cononsingularity, called 
$\mathcal{E}$-$M$-$\mathcal{K}$-nonsingularity and $\mathcal{E}$-$M$-$\mathcal{K}$-cononsingularity respectively. 
We show that if $N$ is (strongly) $M$-CS-Baer and $\mathcal{E}$-$M$-$\mathcal{K}$-cononsingular, then $M$ is (strongly) extending. We also prove that if an object $M$ is (strongly) extending and $\mathcal{E}$-$M$-$\mathcal{K}$-nonsingular, then $M$ is (strongly) self-CS-Baer and $\mathcal{E}$-$M$-$\mathcal{K}$-cononsingular. 
We characterize dual (strongly) self-CS-Baer rings as (strongly) lifting rings, or equivalently, (abelian) semiperfect rings.

Sections 6 and 7 study the relationship between (strongly) relative CS-Baer objects and two relevant generalizations, namely objects with certain summand intersection properties called (strictly) ESSIP and (strictly) SSIP-extending, and (strongly) relative CS-Rickart objects. We prove that every (strongly) self-CS-Baer object is (strictly) ESSIP, and if $A$ and $B$ are objects such that $A\oplus B$ is (strictly) ESSIP, then $B$ is (strongly) $A$-CS-Baer. Also, if $M$ is (strongly) self-CS-Rickart and (strictly) SSIP-extending, then $M$ is (strongly) self-CS-Baer, while the converse holds if ${\rm Soc}(M)$ is an essential subobject of $M$. In particular, if $M$ is a finitely cogenerated right $R$-module or $R$ is a right semiartinian ring, then $M$ is (strongly) self-CS-Baer if and only if $M$ is (strongly) self-CS-Rickart and (strictly) SSIP-extending. Also, if $M$ is a finitely generated right $R$-module or $R$ is a right max ring, then $M$ is dual (strongly) self-CS-Baer if and only if $M$ is dual (strongly) self-CS-Rickart and (strictly) SSSP-lifting.

Section 8 studies products and coproducts of relative CS-Baer objects. The class of (strong) relative CS-Baer objects is closed under direct summands, but in general not under coproducts. Nevertheless, we show that $\bigoplus_{i=1}^n N_i$ is (strongly) $M$-CS-Baer if and only if each $N_i$ is (strongly) $M$-CS-Baer. As a consequence, if $M$ is a right $R$-module, then we deduce that $Z_2(M)$ and $M/Z_2(M)$ are (strongly) $M$-CS-Baer if and only if $Z_2(M)$ is a direct summand of $M$ and $M$ is (strongly) self-CS-Baer. If $M=\bigoplus_{i=1}^n M_i$ is a direct sum decomposition such that $\Hom_{\mathcal{A}}(M_i,M_j)=0$ for every distinct $i,j\in \{1,\dots ,n\}$, we prove that $M$ is (strongly) self-CS-Baer if and only if each $M_i$ is (strongly) self-CS-Baer. 

Finally, in Section 9 we determine the structure of dual (strongly) self-CS-Baer modules over Dedekind domains. We first  show that the problem can be reduced to the case of modules over valuation rings. Then, for a discrete valuation ring $R$ with maximal ideal $\mathfrak{m}$, quotient field $K$ and $Q=K/R$, we prove that an $R$-module $M$ is dual self-CS-Baer if and only if $M$ is isomorphic to one of the following modules: 

(1) $K^a \oplus Q^b \oplus R^c$ with $a \leq 1$ if $R$ is incomplete, or 

(2) $K^{(I_1)} \oplus Q^{(I_2)} \oplus B(n)$, or 

(3) $K^{(I_1)} \oplus B(n, n+1)$, 

\noindent where $I_1$ and $I_2$ are two index sets, $a$, $b$, $c$ and $n$ are non-negative integers, and for natural numbers $n_1$, $n_2$, $\ldots$, $n_s$, $B(n_1, n_2, \ldots, n_s)$ denotes the direct sum of arbitrarily many copies of $R/\mathfrak{m}^{n_1}$, $R/\mathfrak{m}^{n_2}$, $\ldots$, $R/\mathfrak{m}^{n_s}$.

\section{Preliminaries}

In what follows we introduce the necessary terminology and notation.

Every morphism $f:M\to N$ in an abelian category $\mathcal{A}$ has a kernel, cokernel, coimage and image, denoted by ${\rm ker}(f):{\rm Ker}(f)\to M$, ${\rm coker}(f):N\to {\rm Coker}(f)$, ${\rm coim}(f):M\to {\rm Coim}(f)$ and ${\rm im}(f):{\rm Im}(f)\to N$ respectively, and we have ${\rm Coim}(f)\cong {\rm Im}(f)$. A morphism $f:A\to B$ is called a \emph{section} (\emph{retraction}) if it has a left (right) inverse. An abelian category is called \emph{AB3} (\emph{AB3*}) if it has arbitrary coproducts (products). Note that AB3 (AB3*) abelian categories have arbitrary sums (intersections). Examples of such categories include module and comodule categories.

Following \cite{CO1}, recall that a monomorphism $i:K\to M$ (or a subobject $K$ of $M$) in an abelian category $\mathcal{A}$ is called \emph{fully invariant} if for every morphism $h:M\to M$, $hi=i\alpha$ for some morphism $\alpha:K\to K$. Dually, an epimorphism $d:M\to C$ in $\mathcal{A}$ (or a factor object $C=M/K$ of $M$) is called \emph{fully coinvariant} if for every morphism $h:M\to M$, $dh=\beta d$ for some morphism $\beta:C\to C$. A short exact sequence $0\to A\stackrel{i}\to B\stackrel{d}\to C\to 0$ in $\mathcal{A}$ is called \emph{fully invariant} if $i$ is fully invariant, or equivalently, $d$ is fully coinvariant \cite[Lemma 2.5]{CO1}. An object $M$ of $\mathcal{A}$ is called \emph{weak duo} if every section $K\to M$ is fully invariant, or equivalently, every retraction $M\to C$ is fully coinvariant \cite[Definition~2.6]{CO1}.

Recall that a monomorphism $f:M \to N$ in an abelian category $\mathcal{A}$ is called \textit{essential} 
if ${\rm Im}(f)$ is an essential subobject of $N$, that is, for any subobject $X$ of $M$, ${\rm Im}(f)\cap X=0$ implies $X=0$. Dually, an epimorphism $f:M \to N$ in $\mathcal{A}$ is called \textit{superfluous} if ${\rm Ker}(f)$ is a superfluous subobject of $M$, that is, for any subobject $X$ of $M$, ${\rm Ker}(f)+X=M$ implies $X=M$. An object is called \emph{uniform} if every non-zero subobject is essential, and \emph{hollow} if every proper subobject is superfluous.

As natural generalizations of (strongly) extending modules \cite{DHSW} and (strongly) lifting modules \cite{CLVW}, an object $M$ of an abelian category $\mathcal{A}$ is called \emph{(strongly) extending} if every subobject of $M$ is essential in a (fully invariant) direct summand of $M$ and \emph{(strongly) lifting} if every subobject $L$ of $M$ \emph{lies above} a (fully invariant) direct summand of $M$, that is, $L$ contains a (fully invariant) direct summand $K$ of $M$ such that $L/K$ is superfluous in $M/K$. 

Following \cite{CK,CO2}, if $M$ and $N$ are objects of an AB3* abelian category $\mathcal{A}$, then $N$ is called \emph{(strongly) $M$-Baer} if for every family $(f_i)_{i\in I}$ of morphisms $f_i:M\to N$, $\bigcap_{i\in I}{\rm Ker}(f_i)$ is a (fully invariant) direct summand of $M$, and \emph{(strongly) self-Baer} if $N$ is (strongly) $N$-Baer. Dually, $N$ is called \emph{dual (strongly) $M$-Baer} if for every family $(f_i)_{i\in I}$ of morphisms $f_i:M\to N$, $\sum_{i\in I}{\rm Im}(f_i)$ is a (fully invariant) direct summand of $N$, and \emph{dual (strongly) self-Baer} if $N$ is dual (strongly) $N$-Baer. Taking a single-element index set $I$, one gets the concepts of \emph{(strongly) relative Rickart} and \emph{dual (strongly) relative Rickart} objects respectively \cite{CK}. Also, $N$ is called \emph{(strongly) $M$-regular} if $N$ is both (strongly) $M$-Rickart and dual (strongly) $M$-Rickart, and \emph{(strongly) self-regular} if it is $N$-regular \cite{CK,CO2}.

Let $M$ and $N$ be objects of an abelian category $\mathcal{A}$. Following \cite[Definition~9.4]{CK}, recall that $N$ is called \emph{$M$-$\mathcal{K}$-nonsingular} if for any morphism $f:M\to N$ in $\mathcal{A}$, ${\rm Ker}(f)$ essential in $M$ implies $f=0$. Also, $M$ is called \emph{$N$-$\mathcal{T}$-nonsingular} if for any morphism $f:M\to N$ in $\mathcal{A}$, ${\rm Im}(f)$ superfluous in $N$ implies $f=0$. If $M$ is $M$-$\mathcal{K}$-nonsingular ({$M$-$\mathcal{T}$-nonsingular), then $M$ is sometimes simply called \emph{$\mathcal{K}$-nonsingular} (\emph{$\mathcal{T}$-nonsingular}). 

Recall that an object $M$ of an abelian category has the \emph{(strong) summand intersection property}, for short SIP (SSIP), if the intersection of any two (family of) direct summands of $M$ is a direct summand of $M$. Every self-Rickart object has SIP \cite[Corollary~3.10]{CK}, while every self-Baer object has SSIP \cite[Corollary 6.4]{CK}. The dual property to SIP (SSIP) is called the \emph{(strong) summand sum property}, for short SSP (SSSP). Other properties involving intersections and sums of direct summands, such as SIP-extending, ESSIP, SSP-lifting and LSSSP will be introduced later on.  

The main concepts of the paper relate as follows. Part of the diagram is constructed by using the above cited results, while the rest will be clarified in our paper. Each of them has a \emph{strong} version or, in case of summand intersection (sum) properties, \emph{strict} version, obtained by using fully invariant direct summands instead of direct summands.

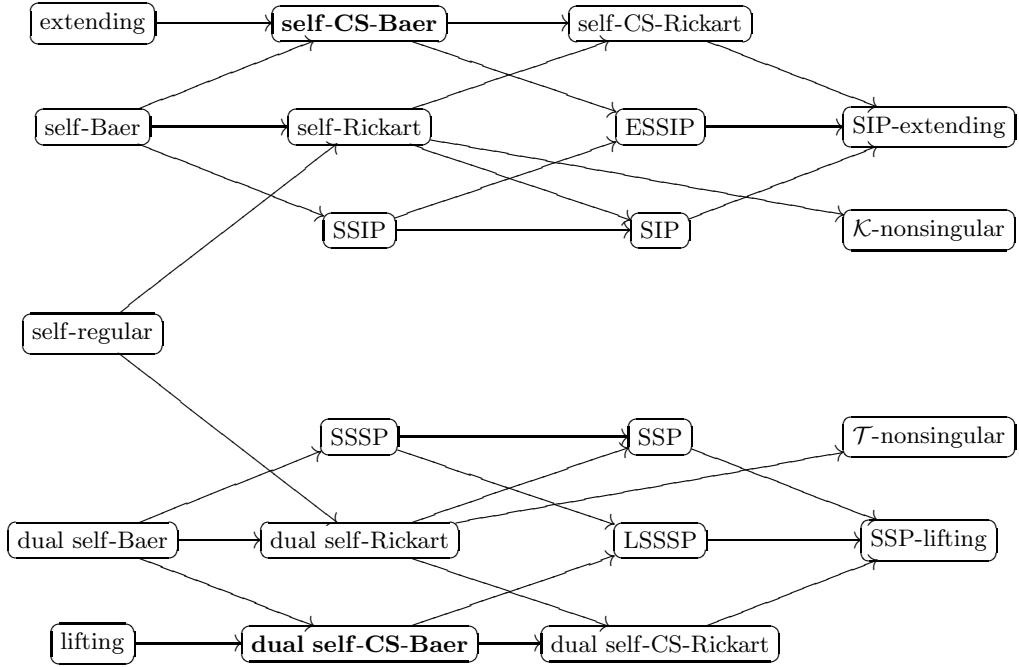
\begin{scriptsize}
$$\SelectTips{cm}{} 
\xymatrix{
*+[F-:<3pt>]{\rm extending} \ar[r] & *+[F-:<3pt>]{\bf self\textrm{-}CS\textrm{-}Baer} \ar[r] \ar[dr] & *+[F-:<3pt>]{\rm self\textrm{-}CS\textrm{-}Rickart} \ar[dr] &  \\
*+[F-:<3pt>]{\rm self\textrm{-}Baer} \ar[ur] \ar[r] \ar[dr] & *+[F-:<3pt>]{\rm self\textrm{-}Rickart} \ar[ur] \ar[dr] \ar[drr] & *+[F-:<3pt>]{\rm ESSIP} \ar[r] & *+[F-:<3pt>]{\rm SIP\textrm{-}extending} \\ 
& *+[F-:<3pt>]{\rm SSIP} \ar[r] \ar[ur] & *+[F-:<3pt>]{\rm SIP} \ar[ur] & *+[F-:<3pt>]{\rm \mathcal{K}\textrm{-}nonsingular}\\
*+[F-:<3pt>]{\rm self\textrm{-}regular} \ar[uur] \ar[ddr] &&& \\
& *+[F-:<3pt>]{\rm SSSP} \ar[r] \ar[dr] & *+[F-:<3pt>]{\rm SSP} \ar[dr] & *+[F-:<3pt>]{\rm \mathcal{T}\textrm{-}nonsingular} \\
*+[F-:<3pt>]{\rm dual \ self\textrm{-}Baer} \ar[ur] \ar[r] \ar[dr] & *+[F-:<3pt>]{\rm dual \ self\textrm{-}Rickart} \ar[dr] \ar[ur] \ar[urr] & *+[F-:<3pt>]{\rm LSSSP} \ar[r] & *+[F-:<3pt>]{\rm SSP\textrm{-}lifting} \\
*+[F-:<3pt>]{\rm lifting} \ar[r] & *+[F-:<3pt>]{\bf dual \ self\textrm{-}CS\textrm{-}Baer} \ar[r] \ar[ur] & *+[F-:<3pt>]{\rm dual \ self\textrm{-}CS\textrm{-}Rickart} \ar[ur] &   \\
}$$
\end{scriptsize}

\section{(Dual) relative CS-Baer objects}

Having prepared the necessary setting, we are ready to define the main notions of the paper as follows. 

\begin{defn} \rm Let $M$ and $N$ be objects of an abelian category $\mathcal{A}$. 
\begin{enumerate}
\item If $\mathcal{A}$ is AB3*, then $N$ is called:
\begin{enumerate}[(i)]
\item \emph{(strongly) $M$-CS-Baer} if for every family $(f_i)_{i\in I}$ of morphisms $f_i:M\to N$, $\bigcap \limits_{i\in I}{\rm Ker}(f_i)$ is essential in a (fully invariant) direct summand of $M$.
\item \emph{(strongly) self-CS-Baer} if $N$ is (strongly) $N$-CS-Baer.
\end{enumerate}
\item If $\mathcal{A}$ is AB3, then $N$ is called:
\begin{enumerate}[(i)]
\item \emph{dual (strongly) $M$-CS-Baer} if for every family $(f_i)_{i\in I}$ of morphisms $f_i:M\to N$, $\sum \limits_{i\in I}{\rm Im}(f_i)$ lies above a (fully invariant) direct summand of $N$.
\item \emph{dual (strongly) self-CS-Baer} if $N$ is dual (strongly) $N$-CS-Baer.
\end {enumerate}
\end{enumerate}
\end{defn}

\begin{rem} \rm (i) Our terminology follows the one in abelian categories from \cite{CK,DNTD}, unlike that of \cite{RR04}, in which a module $M$ is called $N$-Baer if for every family $(f:M\to N)_{i\in I}$ of homomorphisms, $\bigcap_{i\in I}{\rm Ker}(f_i)$ is a direct summand of $M$, and dual $N$-Baer if for every family $(f:M\to N)_{i\in I}$ of homomorphisms, $\sum_{i\in I}{\rm Im}(f_i)$ is a direct summand of $N$. Also, (dual) self-Baer modules are simply called (dual) Baer in \cite{RR04}. 

(ii) When taking a single-element index set $I$ in the above definition, one obtains the concepts of (strongly) relative CS-Rickart and dual (strongly) relative CS-Rickart objects \cite{CR1,CR3}. 

(iii) Let $\A$ be an AB3* abelian category. Then every (strongly) self-Baer object and every (strongly) extending object of $\A$ is (strongly) self-CS-Baer, and every (strongly) relative CS-Baer object is (strongly) relative CS-Rickart. An object $N$ of $\A$ is $M$-CS-Baer for every uniform object $M$ of $\A$. We leave to the reader the statements of the dual remarks.
\end{rem}

(Dual) relative CS-Baer objects are related to (dual) strongly relative CS-Baer objects as follows. 

\begin{prop}\label{st0}
Let $M$ and $N$ be two objects of an abelian category $\mathcal{A}$.
    \begin{enumerate}
       \item Assume that $\mathcal{A}$ is AB3*, and every direct summand of $M$ is 
               isomorphic to a subobject of $N$. Then $N$ is strongly   
               $M$-CS-Baer if and only if $N$ is $M$-CS-Baer and  
               $M$ is weak duo.
       \item Assume that $\mathcal{A}$ is AB3, and every direct summand of $N$ is   
                isomorphic to a factor object of $M$. Then $N$ is   
                dual strongly $M$-CS-Baer if and only if $N$ is  
                dual  $M$-CS-Baer and $N$ is weak duo.
   \end{enumerate}
\end{prop}

\begin{proof} (1) Since  $N$ is strongly $M$-CS-Baer, it is both $M$-CS-Baer and strongly $M$-CS-Rickart, hence $M$ is weak duo \cite[Proposition~2.5]{CR3}. 

Conversely, assume that $N$ is $M$-CS-Baer and $M$ is weak duo. Let $(f_i)_{i\in I}$ be a family of morphisms $f_i:M\to N$ in $\mathcal{A}$. Then there  exists a direct summand $U$ of $M$ such that $\bigcap\limits_{i\in I}{\rm Ker}(f_i)$ is essential in $U$. By the fact that $M$ is weak duo, we deduce that $U$ is a fully invariant direct summand of $M$, and therefore $N$ is strongly $M$-CS-Baer.
\end{proof}

The following corollary generalizes the module-theoretic result \cite[Theorem~2]{Nhan}, and often will be implicitly used throughout the paper.

\begin{coll}\label{st00}
Let $M$ be an object of an abelian category $\mathcal{A}$.
\begin{enumerate}
\item Assume that $\mathcal{A}$ is AB3*. Then the following are equivalent:
\begin{enumerate}[(i)] 
\item $M$ is strongly self-CS-Baer.
\item $M$ is self-CS-Baer and weak duo.
\item $M$ is self-CS-Baer and ${\rm End}_{\mathcal{A}}(M)$ is abelian.
\end{enumerate}
\item Assume that $\mathcal{A}$ is AB3. Then the following are equivalent:
\begin{enumerate}[(i)]
\item $M$ is dual strongly self-CS-Baer.
\item $M$ is dual self-CS-Baer and weak duo.
\item $M$ is dual self-CS-Baer and ${\rm End}_{\mathcal{A}}(M)$ is abelian.
\end{enumerate}
\end{enumerate}
\end{coll}

\begin{proof} (1) (i)$\Leftrightarrow$(ii) This holds by Proposition \ref{st0}.

(ii)$\Rightarrow$(iii) If $M$ is strongly self-CS-Baer, then $M$ is self-CS-Baer and ${\rm End}_{\A}(M)$ is abelian by \cite[Proposition 2.8]{CR3}. 

(iii)$\Rightarrow$(ii) If $M$ is self-CS-Baer and ${\rm End}_{\A}(M)$ is abelian, then $M$ is self-CS-Rickart and ${\rm End}_{\A}(M)$ is abelian, and so $M$ is strongly self-CS-Rickart by \cite[Proposition 2.8]{CR3}. Finally, $M$ is weak duo by \cite[Corollary 2.6]{CR3}.
\end{proof}

\begin{coll}\label{st01}
Let $M$ be an indecomposable object of an abelian category $\mathcal{A}$.
           \begin{enumerate}
                    \item If $\mathcal{A}$ is AB3*, then $M$ is strongly self-CS-Baer if and only if $M$  
                             is self-CS-Baer.
                    \item If $\mathcal{A}$ is AB3, then $M$ is dual strongly self-CS-Baer if and only if 
                             $M$ is dual self-CS-Baer.
\end{enumerate}
\end{coll}

Next we give some examples and non-examples of (dual) relative CS-Baer and (dual) strongly relative CS-Baer objects. Further examples in connection with the other concepts of the paper will be given in the corresponding sections later on.

\begin{ex} \rm We first consider some examples in the category of abelian groups. Let $n$ be a positive integer and denote $\mathbb{Z}_n=\mathbb{Z}/n\mathbb{Z}$. 
         
(i) $\mathbb{Z}$ is strongly self-CS-Baer, but not dual (strongly) self-CS-Rickart (\cite[Example~2.16]{Tribak} or \cite[Example~2.6 (ii)]{CR1}), and thus $\mathbb{Z}$ is not dual self-CS-Baer.
               
(ii) $\mathbb{Z}$ is strongly $\mathbb{Z}_n$-CS-Baer and dual strongly $\mathbb{Z}_n$-CS-Baer, because ${\rm Hom}_{\mathbb{Z}}(\mathbb{Z}_n, \mathbb{Z})=0$.

(iii) For relatively prime integers $m$ an $n$, $\mathbb{Z}_n$ is strongly $\mathbb{Z}_m$-CS-Baer, because ${\rm Hom}_{\mathbb{Z}}(\mathbb{Z}_m,\mathbb{Z}_n)=0$.

(iv) Let $M$ be an extending weak duo (e.g., uniform) object and let $N$ be a lifting weak duo (e.g., hollow) object in an abelian category. Then $N$ is strongly $M$-CS-Baer and dual strongly $M$-CS-Baer. For example, $\mathbb{Z}_4$ is strongly $\mathbb{Z}_6$-CS-Baer and dual strongly $\mathbb{Z}_6$-CS-Baer.

(v) $\mathbb{Z}\oplus \mathbb{Z}_p$ (for some prime $p$) is not strongly self-CS-Baer, because it is not strongly self-CS-Rickart \cite[Example~2.9 (v)]{CR3}. 

(vi) $M=\mathbb{Z}\oplus \mathbb{Z}_2$ is self-CS-Baer, but not strongly self-CS-Baer.
Indeed, for any $f\in {\rm End}_{\mathbb{Z}}(M)$, ${\rm Ker}(f)\in \left\{0, \mathbb{Z}\oplus 0, 0\oplus \mathbb{Z}_2, \mathbb{Z}\oplus \mathbb{Z}_2, K , 2\mathbb{Z}\oplus \mathbb{Z}_2\right\}$, where $K=\{(m,m+2\mathbb{Z})| m \in \mathbb{Z}\}$. Then for any family $(f_i)_{i\in I}$ of endomorphisms of $M$, we have $$\bigcap_{i\in I}{\rm Ker}(f_i)\in \left\{0, 2\mathbb{Z}\oplus 0, \mathbb{Z}\oplus 0, 0\oplus \mathbb{Z}_2, \mathbb{Z}\oplus \mathbb{Z}_2, K , 2\mathbb{Z}\oplus \mathbb{Z}_2\right\}.$$ Now it is easily checked that $M$ is self-CS-Baer. Finally, it is not strongly self-CS-Baer by  \cite[Example~2.9 (v)]{CR3}.
\end{ex}

\begin{ex} \label{ex2} \rm Now we consider some examples in module categories. 

(i) Consider the ring $R=\begin{pmatrix}\mathbb{Z}&\mathbb{Z}\\0&\mathbb{Z} \end{pmatrix}$. 
Then $R$ is neither a self-CS-Baer right $R$-module nor a dual self-CS-Baer right $R$-module, because it is neither a self-CS-Rickart right $R$-module nor a dual self-CS-Rickart right $R$-module \cite [Example~2.7]{CR1}.

(ii) Let $K$ be a field and consider the ring $R=\begin{pmatrix}K&K[X]\\0&K[X] \end{pmatrix}$. 
The direct summands of the right $R$-module $R$ are exactly the following right ideals of $R$: 
$$\begin{pmatrix}0&0\\0&0 \end{pmatrix}, \begin{pmatrix}K&K[X]\\0&K[X] \end{pmatrix}, 
\begin{pmatrix}K&K[X]\\0&0 \end{pmatrix}, \begin{pmatrix}0&f\\0&1 \end{pmatrix}K[X] \quad (f\in K[X]).$$
Every endomorphism of the right $R$-module $R$ is of the form $t_a:R\to R$ given by $t_a(b)=ab$ 
for some $a\in R$, and
\[{\rm Ker}(t_a)\in \left\{\begin{pmatrix}0&0\\0&0 \end{pmatrix}, \begin{pmatrix}0&f\\0&1 \end{pmatrix}K[X] , \begin{pmatrix}K&K[X]\\0&0 \end{pmatrix}, R \Big | f\in K[X]\right\}.\]
Thus, for any family $(f_i)_{i\in I}$ of $R$-homomorphisms $f_i \in \End_R(R)$ we have:
$$\bigcap\limits_{i \in I}{\rm Ker}(f_i)\in \left\{\begin{pmatrix}0&0\\0&0 \end{pmatrix}, \begin{pmatrix}0&f\\0&1 \end{pmatrix}K[X] , \begin{pmatrix}K&K[X]\\0&0 \end{pmatrix}, R  \Big | f\in K[X]\right\},$$
and therefore $R$ is a self-CS-Baer right $R$-module. It is not a strongly self-CS-Baer right $R$-module, because ${\rm End}_R(R)\cong R$ is not abelian. Note also that the right $R$-module $R$ is not dual (strongly) self-CS-Baer, because it is not dual (strongly) self-CS-Rickart \cite [Example~2.8]{CR1} (\cite [Example 2.9]{CR3}).
\end{ex}

\section{(Dual) relative CS-Baer versus (dual) relative Baer}

Every (dual) relative Baer object of an abelian category is clearly (dual) relative CS-Baer, but the converse does not hold in general, as the following examples show.

\begin{ex} \rm 
(i) Let $n > 1$ be an integer and let $p$ be a prime. Let $M_1=\mathbb{Z}_{p^n}$ and $M_2=\mathbb{Z}_{p^n} \oplus \mathbb{Z}_{p^{n+1}}$. Then $M_1$ and $M_2$ are extending and weak duo (see \cite[Proposition A.12]{MM} and \cite[Theorem 3.10]{OHS}), and hence they are strongly self-CS-Baer by Corollary \ref{st00}. However, $M_1$ and $M_2$ are not strongly self-Baer by \cite[Corollary~6.7]{CO2}. 

(ii) Following \cite[Example 3.13]{LRR10}, consider the ring $A=\begin{pmatrix} \mathbb{Z}_2 & \mathbb{Z}_2 \\
0 & \mathbb{Z}_2 \end{pmatrix}$, and the subgroup $G=\{1,g\}$ of ${\rm Aut}(A)$, where $g$ is the conjugation by $\begin{pmatrix} 1 & 1 \\ 0 & 1 \end{pmatrix}$. Consider the skew group ring $R=A\ast G$ and the right $R$-module $M=A$. The endomorphisms of $M$ are given by:
\[\varphi_1=\begin{pmatrix}
0 & 0 \\
0 & 0 \\
\end{pmatrix}, \varphi_2=\begin{pmatrix}
0 & 1 \\
0 & 0 \\
\end{pmatrix}, \varphi_3=\begin{pmatrix}
1 & 0 \\
0 & 1 \\
\end{pmatrix}, \varphi_4=\begin{pmatrix}
1 & 1 \\
0 & 1 \\
\end{pmatrix},\] 
which have the following kernels: 
\[{\rm Ker} (\varphi_1)=M, \quad {\rm Ker} (\varphi_2)=\begin{pmatrix}
\mathbb{Z}_2 & \mathbb{Z}_2 \\
0 & 0 \\
\end{pmatrix}, \quad {\rm Ker} (\varphi_3)=0, \quad {\rm Ker} (\varphi_4)=\begin{pmatrix}
0 & \mathbb{Z}_2 \\
0 & 0 \\
\end{pmatrix}\] \cite[Example~2.3]{CK18}. One may easily notice that $M$ is a self-CS-Baer right $R$-module with abelian endomorphism ring, hence it is a strongly self-CS-Baer right $R$-module by Corollary \ref{st00}. But it is not strongly self-Baer, because it is not self-Rickart \cite[Example 3.13]{LRR10}.

(iii) The $\mathbb{Z}$-module $\mathbb{Z}_n$ (where $n>1$ is an integer) is dual strongly self-CS-Baer (see Corollary \ref{strong-d-Baer} below), but $\mathbb{Z}_n$ is dual strongly self-Baer if and only if it is semisimple by \cite[Theorem 3.4]{KT}. Also, $\mathbb{Z}_n$ is strongly $\mathbb{Z}$-CS-Baer, but in general it is not strongly $\mathbb{Z}$-Baer.  

(iv) Consider the $\mathbb{Z}$-module $\mathbb{Z}(p^{\infty})$ for some prime $p$, and denote by $S$ its endomorphism ring. Then $S$ is an indecomposable dual self-CS-Rickart right $S$-module which is not dual self-Rickart \cite[Example ~2.13]{Tribak}. Hence it is a dual self-CS-Baer right $S$-module, but not dual self-Baer.
\end{ex}

The following lemma on relative nonsingularity will be useful. 

\begin{lemm} \label{l:nonsing}  Let $M$ and $N$ be objects of an abelian category $\mathcal{A}$. 
\begin {enumerate}
\item Assume that $\mathcal{A}$ is AB3*. Then $N$ is $M$-$\mathcal{K}$-nonsingular if and only if $N^I$ is $M$-$\mathcal{K}$-nonsingular for every set $I$.
\item Assume that $\mathcal{A}$ is AB3. Then $N$ is $M$-$\mathcal{T}$-nonsingular if and only if $N$ is $M^{(I)}$-$\mathcal{T}$-nonsingular for every set $I$.
\end{enumerate}
\end{lemm}

\begin{proof} (1) Assume that $N$ is $M$-$\mathcal{K}$-nonsingular. Let $I$ be a set and let $f:M\to N^I$ be a morphism in $\A$ such that ${\rm Ker}(f)$ essential in $M$. For every $i\in I$, denote by $p_i:N^I\to N$ the canonical projection and $f_i=p_if:M\to N$. Since $\bigcap_{i\in I}{\rm Ker}(f_i)={\rm Ker}(f)$ is essential in $M$, so is ${\rm Ker}(f_i)$ for every $i\in I$. By hypothesis, we have $f_i=0$ for every $i\in I$, and thus $f=0$. Hence $N^I$ is $M$-$\mathcal{K}$-nonsingular.
The converse is clear. 
\end{proof}

\begin{theo} \label{t:nonsing} Let $M$ and $N$ be objects of an abelian category $\mathcal{A}$. 
\begin {enumerate}
\item Assume that $\mathcal{A}$ is AB3*. Then $N$ is (strongly) $M$-CS-Baer and $N$ is $M$-$\mathcal{K}$-nonsingular if and only if $N$ is (strongly) $M$-Baer.
\item Assume that $\mathcal{A}$ is AB3. Then $N$ is dual (strongly) $M$-CS-Baer and $N$ is $M$-$\mathcal{T}$-nonsingular if and only if $N$ is dual (strongly) $M$-Baer.
\end{enumerate}
\end{theo}

\begin {proof} (1) By the forthcoming Proposition \ref{csr-csb} and Lemma \ref{l:nonsing}, $N$ is (strongly) $M$-CS-Baer and $N$ is $M$-$\mathcal{K}$-nonsingular if and only if $N^I$ is (strongly) $M$-CS-Rickart and $N^I$ is $M$-$\mathcal{K}$-nonsingular for any set $I$. This is equivalent to $N^I$ being (strongly) $M$-Rickart by \cite[Theorem 2.11]{CR1} and \cite[Theorem 2.12]{CR3}, and furthermore, to $N$ being (strongly) $M$-Baer by \cite[Lemma 6.2]{CK} and \cite[Lemma 5.8]{CO2}. 
\end{proof}

The following corollary generalizes the module-theoretic result \cite[Theorem~1]{Nhan}.

\begin{coll} Let $M$ be an object of an abelian category $\mathcal{A}$. 
\begin {enumerate}
\item Assume that $\mathcal{A}$ is AB3*. Then $M$ is (strongly) self-CS-Baer and $M$-$\mathcal{K}$-nonsingular if and only if $M$ is (strongly) self-Baer.
\item Assume that $\mathcal{A}$ is AB3. Then $M$ is dual (strongly) self-CS-Baer and $M$-$\mathcal{T}$-nonsingular if and only if $M$ is dual (strongly) self-Baer.
\end{enumerate}
\end{coll}

Following \cite[8.1]{CLVW}, \cite[Section~4]{DHSW} and \cite{TV}, let $\mathcal{U}$ be the class of singular right $R$-modules, that is, modules isomorphic to $L/K$ for some module $L$ and essential submodule $K$ of $L$. Also, let $\mathcal{V}$ be the class of small right $R$-modules, that is, modules superfluous in some module. For any right $R$-module $M$, let $Z(M)={\rm Tr}(\mathcal{U},M)$, where
\[{\rm Tr}(\mathcal{U},M)=\sum\{{\rm Im}(f)\mid f\in \Hom(U,M) \textrm{ for some } U\in \mathcal{U}\}\] is the \emph{trace} of $\mathcal{U}$ in $M$, and let $\overline{Z}(M)={\rm Re}(M,\mathcal{V})$, where 
\[{\rm Re}(M,\mathcal{V})=\bigcap\{{\rm Ker}(f)\mid f\in \Hom(M,V) \textrm{ for some } V\in \mathcal{V}\}\] is the \emph{reject} of $\mathcal{V}$ in $M$. A right $R$-module $M$ is called \emph{non-singular} if $Z(M)=0$ and \emph{non-cosingular} if $\overline{Z}(M)=M$. 

\begin{coll} \label{c:Z} Every non-singular (strongly) self-CS-Baer right $R$-module is (strongly) self-Baer and every non-cosingular dual (strongly) self-CS-Baer right $R$-module is dual (strongly) self-Baer.
\end{coll}

\begin{proof} Note that every non-singular right $R$-module $M$ is $M$-$\mathcal{K}$-nonsingular, and every non-cosingular right $R$-module $M$ is $M$-$\mathcal{T}$-nonsingular. Then use Theorem \ref{t:nonsing}. 
\end{proof}

\section{(Dual) relative CS-Baer versus extending (lifting)}

In this section we are interested in relating the (dual) relative CS-Baer property with the extending (lifting) property. Note that every (dual) self-Rickart object and every extending (lifting) object of an abelian category $\A$ is (dual) self-CS-Rickart by definitions. 

\begin{ex} \rm (i) We have seen that the $\mathbb{Z}$-module $\mathbb{Z}\oplus \mathbb{Z}_2$ is self-CS-Baer, but not weak duo. It is neither extending (see \cite[p. 19]{MM}), nor strongly extending.

(ii) The $\mathbb{Z}$-module $\mathbb{Q}$ is dual strongly self-CS-Baer, but it is not (strongly) lifting. Note that for any $0\neq f\in {\rm End}_{\mathbb{Z}}(\mathbb{Q})$, ${\rm Im}(f)=\mathbb{Q}$ is indecomposable.
\end{ex}

If $M$ and $N$ are objects of an abelian category $\mathcal{A}$, then $N$ is $M$-Baer and $M$-$\mathcal{K}$-cononsingular if and only if $M$ is extending and $N$ is $M$-$\mathcal{K}$-nonsingular \cite[Theorem~9.5]{CK}. In order to obtain some similar results on (dual) relative CS-Baer objects, we first recall some notation and introduce some suitable specializations of relative (co)-nonsingularity. 

For objects $M$ and $N$ of an abelian category $\mathcal{A}$, denote $U=\Hom_{\mathcal{A}}(M,N)$.

For every subobject $X$ of $M$ and every subobject $Z$ of $U$, we denote: 
\[l_U(X)=\{f\in U\mid X\subseteq {\rm Ker}(f)\}, \quad r_M(Z)=\bigcap_{f\in Z}{\rm Ker}(f).\]

For every subobject $Y$ of $N$ and every subobject $Z$ of $U$, we denote: 
\[l'_U(Y)=\{f\in U\mid {\rm Im}(f)\subseteq Y\}, \quad r'_N(Z)=\sum_{f\in Z}{\rm Im}(f).\]

\begin{defn} \rm Let $M$ and $N$ be objects of an abelian category $\mathcal{A}$. 

Then $N$ is called:
\begin{enumerate}
\item \emph{$\mathcal{E}$-$M$-$\mathcal{K}$-nonsingular} if for any morphism $f:M\to N$ in $\mathcal{A}$, ${\rm Ker}(f)$ essential in a direct summand of $M$ implies $f=0$. 
\item \emph{$\mathcal{E}$-$M$-$\mathcal{K}$-cononsingular} if for any subobjects $X$ and $Y$ of $M$ such that $X\subseteq Y$, $l_U(X)=l_U(Y) $ implies that $X$ is essential in $Y$.  
\end{enumerate}

Then $M$ is called:
\begin{enumerate}
\item \emph{$\mathcal{L}$-$N$-$\mathcal{T}$-nonsingular} if for any morphism $f:M\to N$ in $\mathcal{A}$, ${\rm Im}(f)$ lies above a direct summand of $N$ implies $f=0$. 
\item \emph{$\mathcal{L}$-$N$-$\mathcal{T}$-cononsingular} if for any subobjects $X$ and $Y$ of $N$ such that $X\subseteq Y$, $l_U'(X)=l_U'(Y)$ implies that $Y$ lies above $X$.    
\end{enumerate}
\end{defn}

\begin{rem} \rm Setting $Y=M$ in the definition of $\mathcal{E}$-$M$-$\mathcal{K}$-cononsingularity and $X=0$ in the definition of $\mathcal{L}$-$N$-$\mathcal{T}$-cononsingularity, one obtains $M$-$\mathcal{K}$-cononsingularity and $N$-$\mathcal{T}$-cononsingularity respectively in the sense of \cite[Definition~9.4]{CK}. Every $\mathcal{E}$-$M$-$\mathcal{K}$-(co)nonsingular object is $M$-$\mathcal{K}$-(co)nonsingular, while every $\mathcal{L}$-$N$-$\mathcal{T}$-(co)nonsingular object is $N$-$\mathcal{T}$-(co)nonsingular. If $N$ is an $M$-Baer object such that $\Hom_{\A}(M,N)\neq 0$, then $N$ is $M$-$\mathcal{K}$-nonsingular by Theorem \ref{t:nonsing}, but not  $\mathcal{E}$-$M$-$\mathcal{K}$-nonsingular. Dually, if $N$ is a dual $M$-Baer object such that $\Hom_{\A}(M,N)\neq 0$, then $M$ is $N$-$\mathcal{T}$-nonsingular by Theorem \ref{t:nonsing}, but not  $\mathcal{L}$-$N$-$\mathcal{T}$-nonsingular. 
\end{rem}

\begin{theo} Let $M$ and $N$ be objects of an abelian category $\mathcal{A}$. 
\begin{enumerate}
\item Assume that $\A$ is AB3*. If $N$ is (strongly) $M$-CS-Baer and $\mathcal{E}$-$M$-$\mathcal{K}$-cononsingular, then $M$ is (strongly) extending.
\item Assume that $\A$ is AB3. If $N$ is dual (strongly) $M$-CS-Baer and $M$ is $\mathcal{L}$-$N$-$\mathcal{T}$-cononsingular, then $N$ is (strongly) lifting.
\end{enumerate}
\end{theo}

\begin{proof} (1) Recal that we denote $U=\Hom_{\mathcal{A}}(M,N)$. Let $L$ be a subobject of $M$ and denote 
$$E=r_M(l_U(L))=\bigcap\{{\rm Ker}(f)| f\in l_U(L)\}=\bigcap\{{\rm Ker}(f)| L\subseteq {\rm Ker}(f), f\in U\}.$$ Then we have $L\subseteq E\subseteq M$. Since $N$ is (strongly) $M$-CS-Baer, $E$ is essential in a (fully invariant) direct summand $K$ of $M$. Write $M=K\oplus K'$ for some subobject $K'$ of $M$. We only need to show that $L$ is essential in $E$. To do this, we consider a subobject $X$ of $E$ such that $X\cap L=0$. By \cite[Theorem~9.1 and Lemma~9.2]{CK} we deduce:
\[l_U(L+K')=l_U(L)\cap l_U(K')=l_U(r_M(l_U(L)))\cap l_U(K')=l_U(r_M(l_U(L))+K')=l_U(E+K').\]
Now, since $N$ is $\mathcal{E}$-$M$-$\mathcal{K}$-cononsingular, it follows that $L+K'$ is essential in $E+K'$. Using the modularity condition, we have:
\[X\cap (L+K')=X\cap [K\cap (L+K')]=X\cap [L+(K\cap K')]=X\cap L=0,\]
whence we deduce that $X=0$. This shows that $L$ is essential in $E$, and consequently, $M$ is (strongly) extending.  
\end{proof}

\begin{coll} \label{c:nonsing} Let $M$ be an object of an abelian category $\mathcal{A}$. 
\begin{enumerate}
\item Assume that $\A$ is AB3*. If $M$ is (strongly) self-CS-Baer and $\mathcal{E}$-$M$-$\mathcal{K}$-cononsingular, then $M$ is (strongly) extending.
\item Assume that $\A$ is AB3. If $M$ is dual (strongly) self-CS-Baer and $\mathcal{L}$-$M$-$\mathcal{T}$-cononsingular, then $M$ is (strongly) lifting.
\end{enumerate}
\end{coll}

\begin{theo} \label{t:khuri} Let $M$ and $N$ be objects of an abelian category $\mathcal{A}$. 
\begin{enumerate}
\item Assume that $\A$ is AB3* and any direct summand of $M$ is isomorphic to a subobject of $N$. If $M$ is (strongly) extending and $N$ is $\mathcal{E}$-$M$-$\mathcal{K}$-nonsingular, then $N$ is (strongly) $M$-CS-Baer and $\mathcal{E}$-$M$-$\mathcal{K}$-cononsingular.
\item Assume that $\A$ is AB3 and any direct summand of $N$ is isomorphic to a factor object of $M$. If $N$ is (strongly) lifting and $M$ is $\mathcal{L}$-$N$-$\mathcal{T}$-nonsingular, then $N$ is dual (strongly) $M$-CS-Baer and $M$ is $\mathcal{L}$-$N$-$\mathcal{T}$-cononsingular.
\end{enumerate}
\end{theo}

\begin{proof} (1) Assume that $M$ is (strongly) extending and $N$ is $\mathcal{E}$-$M$-$\mathcal{K}$-nonsingular. Then clearly $N$ is (strongly) $M$-CS-Baer. In order to prove that $N$ is $\mathcal{E}$-$M$-$\mathcal{K}$-cononsingular, we consider subobjects $X$ and $Y$ of $M$ such that $X\subseteq Y$ and  $l_U(X)=l_U(Y)$. Since $M$ is (strongly) extending, $X$ is essential in a (fully invariant) direct summand $K$ of $M$, where $M=K\oplus K'$ for some subobject $K'$ of $M$. Denote by $p:M\to K'$ the projection of $M$ onto $K'$, let $i:K'\to N$ be a monomorphism and consider $f=ip:M\to N$. Since $M$ is extending, ${\rm Ker}(f)=K$ is essential in a direct summand of $M$. But $N$ is $\mathcal{E}$-$M$-$\mathcal{K}$-nonsingular, whence we get $f=0$. Now $X$ is esssential in $M$, and consequently, $X$ is essential in $Y$. Hence $N$ is $\mathcal{E}$-$M$-$\mathcal{K}$-cononsingular.
\end{proof}

\begin{coll} Let $M$ be an object of an abelian category $\mathcal{A}$. 
\begin{enumerate}
\item Assume that $\A$ is AB3*. If $M$ is (strongly) extending and $\mathcal{E}$-$M$-$\mathcal{K}$-nonsingular, then $M$ is (strongly) self-CS-Baer and $\mathcal{E}$-$M$-$\mathcal{K}$-cononsingular.
\item Assume that $\A$ is AB3. If $M$ is (strongly) lifting and $\mathcal{L}$-$M$-$\mathcal{T}$-nonsingular, then $M$ is dual (strongly) self-CS-Baer and $\mathcal{L}$-$M$-$\mathcal{T}$-cononsingular.
\end{enumerate}
\end{coll}

Now we may deduce the following known result on relative Baer objects, which is part of \cite[Theorem~9.5]{CK} (\cite[Theorem~5.17]{CO2}).

\begin{coll} Let $M$ and $N$ be objects of an abelian category $\mathcal{A}$. 
\begin{enumerate}
\item Assume that $\A$ is AB3* and any direct summand of $M$ is isomorphic to a subobject of $N$. If $M$ is (strongly) extending and $N$ is $M$-$\mathcal{K}$-nonsingular, then $N$ is (strongly) $M$-Baer and $M$-$\mathcal{K}$-cononsingular.
\item Assume that $\A$ is AB3 and any direct summand of $N$ is isomorphic to a factor object of $M$. If $N$ is (strongly) lifting and $M$ is $N$-$\mathcal{T}$-nonsingular, then $N$ is dual (strongly) $M$-Baer and $M$ is $N$-$\mathcal{T}$-cononsingular.
\end{enumerate}
\end{coll}

\begin{proof} Use the same proof as for Theorem \ref{t:khuri} with $Y=M$ together with Theorem \ref{t:nonsing}.
\end{proof}

We continue with the following module-theoretic result, which also shows that the dual (strong) self-CS-Baer property is left-right symmetric for rings. 

\begin{prop} \label{p:semiperfect} Let $R$ be a unitary ring. The following are equivalent: 
\begin{enumerate}[(i)]
\item $R$ is a dual (strongly) self-CS-Baer right $R$-module.
\item $R$ is a (strongly) lifting right $R$-module.
\item $R$ is an (abelian) semiperfect ring.  
\end{enumerate}
\end{prop}

\begin{proof} (i)$\Rightarrow$(ii) Let $I$ be a right ideal of $R$. Then for all $a\in R$, consider the $R$-homomorphisms $\varphi_a:R\to R$ defined by $\varphi_a(r)=ar$. Note that ${\rm Im}(\varphi_a)=aR$ for all $a\in I$, and $\sum_{a\in I} {\rm Im}(\varphi_a)=I$. Since $R$ is dual (strongly) self-CS-Baer, there exists a (fully invariant) direct summand $K$ of $R$ such that $I$ lies above $K$. Hence $R$ is (strongly) lifting.

(ii)$\Rightarrow$(i) This is clear.

(ii)$\Leftrightarrow$(iii) This follows by \cite[Corollary~4.42]{MM} (and the fact that $R$ is weak duo if and only if $R$ is abelian).
\end{proof}

We also give a related characterization for dual (strongly) self-CS-Rickart rings, which shows that the dual (strong) self-CS-Rickart property is left-right symmetric for rings (see \cite[Proposition~2.12]{Tribak} for a different proof). Recall that a unitary ring $R$ is called \emph{semiregular} if every left (right) ideal of $R$ lies above a direct summand. 

\begin{prop} Let $R$ be a unitary ring. Then $R$ is a dual (strongly) self-CS-Rickart right $R$-module if and only if $R$ is an (abelian) semiregular ring. 
\end{prop}

\begin{proof} Assume first that $R$ is a dual (strongly) self-CS-Rickart right $R$-module. Let $a\in R$, and consider the right ideal $I=aR$ of $R$ and the $R$-homomorphism $\varphi_a:R\to R$ defined by $\varphi_a(r)=ar$. Since $R$ is dual (strongly) self-CS-Rickart, there exists a (fully invariant) direct summand $K$ of $R$ such that $I=aR={\rm Im}(\varphi_a)$ lies above $K$ (and $R$ is weak duo). Hence $R$ is an (abelian) semiregular ring.

Conversely, assume that $R$ is an (abelian) semiregular ring. Let $f:R\to R$ be an $R$-homomorphism, and denote $a=f(1)$. Since $R$ is an (abelian) semiregular ring, $aR={\rm Im}(f)$ lies above a direct summand of $R$ (and $R$ is weak duo). Hence $R$ is a dual (strongly) self-CS-Rickart right $R$-module.
\end{proof}

\begin{rem} \rm Since there are semiregular rings that are not semiperfect, a projective dual (strongly) self-CS-Rickart module need not be dual (strongly) self-CS-Baer. 
\end{rem}

\begin{prop} Let $R$ be a commutative unitary ring. Then every cyclic dual (strongly) self-CS-Baer $R$-module is (strongly) lifting. 
\end{prop}

\begin{proof} Let $M=yR$, and let $N$ be a submodule of $M$. Then $N=\sum_{x\in N}xR$. Now let $x\in N$. Then $x=yr$ for some $r\in R$. Consider the $R$-homomorphism $\varphi_x:M\to M$ defined by $\varphi_x(yt)=xt$ for all $t\in R$. Here ${\rm Im}(\varphi_x)=xR$, hence $N=\sum_{x\in N}{\rm Im}(\varphi_x)$. Since $M$ is dual (strongly) self-CS-Baer, $N$ lies above a (fully invariant) direct summand of $M$. Hence $M$ is (strongly) lifting.
\end{proof}

\section{(Dual) relative CS-Baer versus ESSIP (LSSSP)}

We shall see in the next section that CS-Baer and CS-Rickart objects may be related by means of some conditions involving direct summands. In order to get there we need now to prepare the setting.

In the study of (strongly) self-CS-Rickart objects, it is useful to consider the following concepts generalizing SIP (SSIP) and SSP (SSSP). SIP-extending (SSIP-extending), strictly SIP-extending (strictly SSIP-extending) objects and their duals in abelian categories, called SSP-lifting (SSSP-lifting), strictly SSP-lifting (strictly SSSP-lifting), were considered in \cite{CR1,CR3}.

\begin{defn} \rm An object $M$ of an abelian category $\A$ with AB3* is called:
\begin{enumerate}
\item \emph{SIP-extending} (\emph{SSIP-extending}) if for any two (family of) subobjects of $M$ which are essential in direct summands of $M$, their intersection is essential in a direct summand of $M$.
\item \emph{strictly SIP-extending} (\emph{strictly SSIP-extending}) if for any two (family of) subobjects of $M$ which are essential in direct summands of $M$, their intersection is essential in a fully invariant direct summand of $M$.
\item \emph{ESIP} (\emph{ESSIP}) if for any two (family of) direct summands of $M$, their intersection is essential in a direct summand of $M$.
\item \emph{strictly ESIP} (\emph{strictly ESSIP}) if for any two (family of) direct summands of $M$, their intersection is essential in a fully invariant direct summand of $M$.
\end{enumerate}
We leave to the reader the definition of the dual concepts, which are called \emph{SSP-lifting (SSSP-lifting)}, \emph{strictly SSP-lifting (strictly SSSP-lifting)}, \emph{LSSP (LSSSP)} and \emph{strictly LSSP (strictly LSSSP)} respectively.  
\end{defn}

\begin{rem} \rm (i) Every (strongly) self-CS-Rickart object is (strictly) SIP-extending \cite[Corollary~3.6]{CR1} (\cite[Corollary~3.6]{CR3}). 

(ii) An object is (strictly) SIP-extending if and only if it is (strictly) ESIP \cite[Lemma ~3.4]{CR1} (\cite[Lemma ~3.4]{CR3}). 

(iii) Every (strictly) SSIP-extending object is (strictly) ESSIP. 
\end{rem}

\begin{lemm} \label{l:socrad} Let $M$ be an object of an abelian category $\mathcal{A}$.
\begin{enumerate}
\item Assume that  $\mathcal{A}$ is AB3* and the socle ${\rm Soc}(M)$ of $M$ is essential in $M$. Then $M$ is (strictly) SSIP-extending if and only if $M$ is (strictly) ESSIP.
\item Assume that $\mathcal{A}$ is AB3 and the radical ${\rm Rad}(M)$ of $M$ is superfluous in $M$. Then $M$ is (strictly) SSSP-lifting if and only if $M$ is (strictly) LSSSP.
\end{enumerate}
\end{lemm}

\begin{proof} (1) We have already seen the direct implication. Conversely, suppose that $M$ is (strictly) ESSIP. Let $(M_i)_{i \in I}$ be a family of subobjects of $M$ such that $M_i$ is essential in a direct summand $L_i$ of $M$ for every $i\in I$. As in the proof of \cite[Theorem~3]{Nhan}, it follows that $\bigcap \limits_{i \in I} {M_i}$ is essential in $\bigcap \limits_{i \in I} {L_i}$. By hypothesis, $\bigcap \limits_{i \in I} {L_i}$ is essential in a (fully invariant) direct summand $L$ of $M$. Then  $\bigcap \limits_{i \in I} {M_i}$ is essential in $L$, which shows that $M$ is (strictly) SSIP-extending. 
\end{proof}

\begin{theo}\label{csb-ss2} Let $M$ and $N$ be objects of an abelian category $\mathcal{A}$.
 \begin{enumerate}
 \item Assume that $\mathcal{A}$ is AB3* and any direct summand of $M$ is isomorphic to a subobject of $N$. If $N$ is (strongly) $M$-CS-Baer, then $M$ is (strictly) $ESSIP$.
\item Assume that $\mathcal{A}$ is AB3 and any direct summand of $N$ is isomorphic to a factor object of $M$. If $N$ is dual (strongly) $M$-CS-Baer, then $N$ is (strictly) LSSSP.
 \end{enumerate}
\end{theo}

\begin{proof} (1) Assume that $N$ is (strongly) $M$-CS-Baer. Let $(M_i)_{i\in I}$ be a family of direct summands of $M$. For every $i\in I$, there is a split short exact sequence 
$$\SelectTips{cm}{} 
\xymatrix{
  0\ar[r] &M_i\ar[r]&M\ar[r]^-{g_i}& M/M_i\ar[r] &0.
}
$$ 
By hypothesis there is a monomorphism $u:\prod\limits_{i\in I}(M/M_i)\to N^I$. For every $j\in I$, consider the morphism $f_j=p_j u\prod\limits_{i\in I}g_i:M\to N$, where $p_j:N^I\to N$ is the canonical projection. Since $N$ is (strongly) $M$-CS-Baer, $\bigcap\limits_{i\in I} M_i=\bigcap\limits_{j\in I} {\rm Ker}(f_j)$ is essential in a (fully invariant) direct summand of $M$. Hence $M$ is (strictly) ESSIP.
\end{proof}

\begin{coll} \label{c:EL} Let $M$ be an object of an abelian category $\mathcal{A}$.
 \begin{enumerate}
 \item Assume that $\mathcal{A}$ is AB3*. Then every (strongly) self-CS-Baer object is (strictly) $ESSIP$.
\item Assume that $\mathcal{A}$ is AB3. Then every dual (strongly) self-CS-Baer object is (strictly) LSSSP.
 \end{enumerate}
\end{coll}

In general the converse of the above corollary does not hold, as we may see below. 

\begin{ex} \rm Let $S$ be a simple right Ore domain which is not a division ring. Consider the ring $R=\begin{pmatrix} S&K/S \\ 0&S\end{pmatrix}$, where $K$ is the classical right ring of quotients of $S$. Take the right $R$-module $M=\begin{pmatrix} S&K/S \\ 0&0\end{pmatrix}$. By \cite[Example~12]{LR17}, we have ${\rm End}_R(M)\cong S$. Hence ${\rm End}_R(M)$ is a simple domain, and therefore $M$ is indecomposable. Then it is clear that $M$ is strictly ESSIP and strictly LSSSP. 

Now let $t\in S$ be a non-zero non-invertible element ($S$ is not a division ring), and $r=\begin{pmatrix} t&0 \\ 0&0\end{pmatrix}$. Consider the $R$-homomorphism $f_t:M\to M$ defined by $f_t(x)=rx$ for all $x\in M$. It is not hard to see that \[0\neq {\rm Ker}(f_t)=\left\{\begin{pmatrix} 0&k+S \\ 0&0\end{pmatrix}\Big| k\in K \textrm{ and } tk\in S\right\},\] and ${\rm Ker}(f_t)$ is not essential in $M$. Remembering that $M$ is indecomposable, $M$ is not (strongly) self-CS-Rickart, and thus not (strongly) self-CS-Baer. On the other hand, we have \[{\rm Im}(f_t)=\left\{\begin{pmatrix} ts&tk+S \\ 0&0\end{pmatrix}\Big| (s,k)\in S\times K\right\}\] by \cite[Example~2.9]{TTH}. Here ${\rm Im}(f_t)$ is not superfluous in $M$, because ${\rm Im}(f_t)\nsubseteq {\rm Rad}(M)$. Note that ${\rm Rad}(M)=J(R)$ (the Jacobson radical of $R$) and $J(S)=0$, because $S$ is a simple ring. Hence $M$ is not dual (strongly) self-CS-Rickart, and thus not dual (strongly) self-CS-Baer. 
\end{ex}

Nevertheless, we have the following property.

\begin{lemm} \label{l:AB} Let $M$ and $N$ be objects of an abelian category $\A$.  
\begin{enumerate}
\item If $\mathcal{A}$ is AB3* and $M\oplus N$ is (strictly) ESSIP, then $N$ is (strongly) $M$-CS-Baer.
\item If $\mathcal{A}$ is AB3 and $M\oplus N$ is (strictly) LSSSP, then $N$ is dual (strongly) $M$-CS-Baer.
\end{enumerate} 
\end{lemm}

\begin{proof} (1) Let $(f_i)_{i\in I}$ be a family of morphisms $f_i:M\to N$ in $\A$. 
The morphisms $\left[\begin{smallmatrix} 1 \\ 0 \end{smallmatrix}\right]:M\to M\oplus N$ 
and $g_i=\left[\begin{smallmatrix} 1 \\ f_i \end{smallmatrix}\right]:M\to M\oplus N$ are sections for every $i\in I$, having obvious left inverses. For each $i\in I$ denote by $j_i:{\rm Im}(g_i)\to M\oplus N$ the inclusion morphism, and construct the left commutative diagram
$$\SelectTips{cm}{}
\xymatrix{
& 0 \ar[d] & 0 \ar[d] & & & \\
0 \ar[r] & {\rm Ker}(f_i) \ar[r] \ar[d] & M \ar[r] \ar[d]^{\left[\begin{smallmatrix} 1 \\ 
0 \end{smallmatrix}\right]} &C \ar@{=}[d] \ar[r] & 0 \\
0 \ar[r] & {\rm Im}(g_i) \ar[r]_{j_i} \ar[d] & 
M\oplus N \ar[r] \ar[d] & C \ar[r] & 0 \\
&{\rm Im}(f_i) \ar@{=}[r] \ar[d] & {\rm Im}(f_i) \ar[d] & & \\ 
& 0 & 0 & & &
}
\hspace{1cm}
\xymatrix{
& \\ 
\bigcap_{i\in I} {\rm Ker}(f_i) \ar[r] \ar[d] & M \ar[d] \\ 
\bigcap_{i\in I} {\rm Im}(g_i) \ar[r] & M\oplus N
}
$$
whose left upper square is both a pullback and a pushout (e.g., from \cite[Proposition~2.12]{Buhler} for the exact structure consisting of all short exact sequences). Then the induced right commutative diagram is a pullback and a pushout by \cite[Lemma~2.7]{CKO21}, where all morphisms are monomorphisms. Since $M\oplus N$ is (strictly) ESSIP, 
$\bigcap_{i\in I} {\rm Ker}(f_i)=M\cap \left(\bigcap_{i\in I} {\rm Im}(g_i)\right)$ is essential in a (fully invariant) direct summand of $M\oplus N$, and thus it is essential in a (fully invariant) direct summand of $M$. Hence $N$ is (strongly) $M$-CS-Baer.
\end{proof}

\begin{coll} Let $M$ be an object of an abelian category $\A$.  
\begin{enumerate}
\item If $\mathcal{A}$ is AB3* and $M\oplus M$ is (strictly) ESSIP, then $M$ is (strongly) self-CS-Baer.
\item If $\mathcal{A}$ is AB3 and $M\oplus M$ is (strictly) LSSSP, then $M$ is dual (strongly) self-CS-Baer.
\end{enumerate} 
\end{coll}

\section{(Dual) relative CS-Baer versus (dual) relative CS-Rickart}

Every (dual) relative CS-Baer object is clearly (dual) relative CS-Rickart. The converse does not hold in general, as we may see in the following example.

\begin{ex} \rm (i) Consider the $\mathbb{Z}$-module $M=\mathbb{Z}^{(\mathbb{R})}$. By \cite[Remark~2.28]{LRR10}, $M$ is not self-Baer, but it is self-Rickart, and hence $M$ is self-CS-Rickart. Since $\mathbb{Z}$ is non-singular, then so is $M$ \cite[Proposition~1.22]{Goodearl}. Now Corollary \ref{c:Z} allows us to deduce that $M$ is not self-CS-Baer. 

(ii) Let $F$ be a field and let $A=F^{\mathbb{N}}$. Let $R$ be the subring of $A$ consisting of sequences $(a_1, a_2, \ldots) \in A$ that are eventually constant. Clearly, $R$ is a commutative ring. Moreover, it is well known that $R$ is a von Neumann regular ring. Hence $R$ is a self-Rickart and dual self-Rickart $R$-module, and consequently, a self-CS-Rickart and dual self-CS-Rickart $R$-module. Now let us prove that $R$ is not a self-CS-Baer $R$-module. For any integer $i \geq 1$, let $e_i \in R$ denote the $i^{\rm th}$ unit vector $(0, \ldots, 1, 0, \ldots)$. Set $S=\{e_1, e_3, e_5, \ldots\}$. Then ${\rm ann}_R(S)$ consists of sequences $a=(a_1, a_2, \ldots)$ which are eventually zero, and such that $a_n=0$ for $n$ odd (see \cite[Example 7.54]{Lam}). Suppose that ${\rm ann}_R(S)$ is essential in $eR$ for some idempotent $e \in R$. Then $e=(b_1, b_2, \ldots)$ such that $b_i = 0$ or $b_i=1$ for all $i \geq 1$. Moreover, $b_i=1$ for $i$ even and there exists $n \geq 1$ such that $b_i=1$ for all $i \geq n$. Consider the $R$-submodule $X$ of $eR$ consisting of sequences $x=(x_1, x_2, \ldots)$ which are eventually zero, and such that $x_n=0$ for $n$ even. It is easily seen that $X \cap {\rm ann}_R(S)=0$, a contradiction. It follows that the $R$-module $R$ is not self-CS-Baer. In addition, note that the $R$-module $R$ is not dual self-CS-Baer, since $R$ is not a semiperfect ring (see Proposition \ref{p:semiperfect}).
\end{ex}

The following proposition shows that one may use the theory of relative CS-Rickart objects in order to develop the theory of relative CS-Baer objects in abelian categories. We illustrate its applications several times throughout the paper. 

\begin{prop}\label{csr-csb}
 \label{CSBR} Let $M$ and $N$ be objects of an abelian category $\mathcal{A}$.
          \begin{enumerate}
                    \item Assume that $\mathcal{A}$ is AB3*. Then $N$ is (strongly) $M$-CS-Baer if and only if $N^I$ is (strongly) $M$-CS-Rickart for any set $I$.
                   \item Assume that $\mathcal{A}$ is AB3. Then $N$ is dual (strongly) $M$-CS-Baer if and only if $N$ is dual (strongly) $M^{(I)}$-CS-Rickart for any set $I$. 
         \end{enumerate}
\end{prop}

\begin{proof} (1) Suppose that $N$ is (strongly) $M$-CS-Baer. Let $I$ be a set and let $f:M\to N^I$ be a morphism in $\mathcal{A}$. Denote by $p_i:N^I\to N$ the canonical projection and $f_i=p_if:M\to N$ for any $i\in I$. Since $N$ is $M$-CS-Baer, ${\rm Ker}(f)=\bigcap_{i\in I} {\rm Ker}(f_i)$ is essential in a (fully invariant) direct summand of $M$, and thus $N^I$ is (strongly) $M$-CS-Rickart.

Conversely, suppose that $N^I$ is (strongly) $M$-CS-Rickart for any set $I$. Let $(f_i)_{i\in I}$ be a family of morphisms $f_i:M\to N$ in $\mathcal{A}$. Then from the universal property of the product, there is a unique morphism $f=\prod\limits_{i\in I}f_i:M\to N^I$ such that for every $i\in I$, we have $p_if=f_i$, where $p_i:N^I\to N$ is the canonical projection. Since $N^I$ is (strongly) $M$-CS-Rickart, $\bigcap_{i\in I} {\rm Ker}(f_i)={\rm Ker}(f)$ is essential in a (fully invariant) direct summand of $M$. Thus $N$ is (strongly) $M$-CS-Baer.
\end{proof}

Next we show how the notions of (strictly) SSIP-extending and (strictly) ESSIP objects are related to that of (strongly) CS-Baer object.

\begin{theo}\label{csb-ss} Let $M$ and $N$ be objects of an abelian category $\mathcal{A}$.
\begin{enumerate}
\item Assume that $\mathcal{A}$ is AB3*. If $N$ is (strongly) $M$-CS-Rickart and $M$ is (strictly) SSIP-extending, then $N$ is (strongly) $M$-CS-Baer. The converse holds if any direct summand of $M$ is isomorphic to a subobject of $N$, and ${\rm Soc}(M)$ is an essential subobject of $M$.
\item Assume that $\mathcal{A}$ is AB3. If $N$ is dual (strongly) $M$-CS-Rickart and $N$ is (strictly) SSSP-lifting, then  $N$ is dual (strongly) $M$-CS-Baer. The converse holds if any direct summand of $N$ is isomorphic to a factor object of $M$, and ${\rm Rad}(M)$ is a superfluous subobject of $M$.
\end{enumerate}
\end{theo}

\begin{proof} (1) Assume that $N$ is (strongly) $M$-CS-Rickart and $M$ is (strictly) SSIP-extending. Let $(f_i)_{i\in I}$ be a family of morphisms $f_i:M\to N$. Since $N$ is (strongly) $M$-CS-Rickart, ${\rm Ker}(f_i)$ is essential in a (fully invariant) direct summand of $M$ for every $i\in I$, and because $M$ is (strictly) SSIP-extending $\bigcap\limits_{i\in I}{\rm Ker}(f_i)={\rm Ker}\left(\prod\limits_{i\in I}f_i\right)$ is essential in a (fully invariant) direct summand of $M$. Hence $N$ is (strongly) $M$-CS-Baer.

The converse follows by Lemma \ref{l:socrad} and Theorem \ref{csb-ss2}.
\end{proof}

The following corollary generalizes the module-theoretic result \cite[Theorem~3]{Nhan}.

\begin{coll} \label{c:socrad} Let $M$ and $N$ be objects of an abelian category $\mathcal{A}$.
 \begin{enumerate}
 \item Assume that $\mathcal{A}$ is AB3*. If $M$ is (strongly) self-CS-Rickart and (strictly) SSIP-extending, then $M$ is (strongly) self-CS-Baer. The converse holds if ${\rm Soc}(M)$ is an essential subobject of $M$. 
\item Assume that $\mathcal{A}$ is AB3. If $M$ is dual (strongly) self-CS-Rickart and (strictly) SSSP-lifting, then $M$ is dual (strongly) self-CS-Baer. The converse holds if ${\rm Rad}(M)$ is a superfluous subobject of $M$.
 \end{enumerate}
\end{coll}

We present some illustrations of the above corollary in module and comodule categories. Recall that a unitary ring $R$ is called \emph{right semiartinian} if every non-zero right $R$-module has non-zero socle, and \emph{right max} if every non-zero right $R$-module has a maximal submodule.

\begin{coll} Let $R$ be a unitary ring, and let $M$ be a right $R$-module. Then:
\begin{enumerate} 
\item Assume that $M$ is finitely cogenerated or $R$ is a right semiartinian ring. Then $M$ is (strongly) self-CS-Baer if and only if $M$ is (strongly) self-CS-Rickart and (strictly) SSIP-extending.
\item Assume that $M$ is finitely generated or $R$ is a right max ring. Then $M$ is dual (strongly) self-CS-Baer if and only if $M$ is dual (strongly) self-CS-Rickart and (strictly) SSSP-lifting. 
\end{enumerate}
\end{coll}

\begin{proof} If $M$ is finitely cogenerated or $R$ is a right semiartinian ring, then ${\rm Soc}(M)$ is an essential submodule of $M$ \cite[21.3]{Wis}. If $M$ is finitely generated or $R$ is a right max ring, then ${\rm Rad}(M)$ is a superfluous submodule of $M$ \cite[21.6]{Wis}. Finally, use Corollary \ref{c:socrad}.
\end{proof}

\begin{coll} Let $C$ be a coalgebra over a field, and let $M$ be a left $C$-comodule. Then: 
\begin{enumerate} 
\item $M$ is (strongly) self-CS-Baer if and only if $M$ is (strongly) self-CS-Rickart and (strictly) SSIP-extending.
\item If $C$ is right semiperfect, then $M$ is dual (strongly) self-CS-Baer if and only if $M$ is dual (strongly) self-CS-Rickart and (strictly) SSSP-lifting.
\end{enumerate}
\end{coll}

\begin{proof} First note that the socle of any non-zero left $C$-comodule $M$ is an essential subcomodule of $M$ \cite[Corollary~2.4.12]{DNR}. Next let $X$ be a left $C$-comodule such that ${\rm Rad}(M)+X=M$. Since $C$ is right semiperfect, we have ${\rm Rad}(M)\neq M$ by \cite[Proposition~3.2.2 and Corollary~3.2.6]{DNR}, and thus $X\neq 0$. Again by \cite[Corollary~3.2.6]{DNR}, $X$ contains a maximal subcomodule, say $Y$. It follows that ${\rm Rad}(M)\subseteq Y\subset X$, which implies $X=M$. Hence ${\rm Rad}(M)$ is a superfluous subcomodule of $M$. Finally, use Corollary \ref{c:socrad}.
\end{proof}

\section{(Co)products of (dual) relative CS-Baer objects}

Now we analyze the behaviour of (strongly) relative CS-Baer objects and their duals with respect to direct summands and (co)products. 

\begin{coll} \label{t:sdr1} Let $r:M\to M'$ be an epimorphism and $s:N'\to N$ a monomorphism in an abelian category $\mathcal{A}$. 
\begin{enumerate}
\item Assume that $\mathcal{A}$ is AB3* and $r$ is a retraction. If $N$ is (strongly) $M$-CS-Baer, then $N'$ is (strongly) $M'$-CS-Baer.
\item Assume that $\mathcal{A}$ is AB3 and $s$ is a section. If $N$ is dual (strongly) $M$-CS-Baer, then $N'$ is dual (strongly) $M'$-CS-Baer.
\end{enumerate}
\end{coll}

\begin{proof}
(1)  If $N$ is (strongly) $M$-CS-Baer, then $N^I$ is (strongly) $M$-CS-Rickart for every set $I$ by Proposition~\ref{csr-csb}. Then $N'^I$ is (strongly) $M'$-CS-Rickart for every set $I$ by \cite [Theorem~3.1]{CR1} (\cite [Theorem~3.1]{CR3}). Therefore, $N'$ is (strongly) $M'$-CS-Baer by Proposition~\ref{csr-csb}.
\end{proof}

The following corollary generalizes the module-theoretic result \cite[Theorem~4]{Nhan}.

\begin{coll} \label{c:sdb5} Let $M$ and $N$ be objects of an abelian category $\mathcal{A}$, 
$M'$ a direct summand of $M$ and $N'$ a direct summand of $N$. 
\begin{enumerate} 
\item Assume that $\mathcal{A}$ is AB3*, and  $N$ is (strongly) $M$-CS-Baer, then $N'$ is (strongly) $M'$-CS-Baer.
\item Assume that $\mathcal{A}$ is AB3, and $N$ is dual (strongly) $M$-CS-Baer, then $N'$ is dual (strongly) $M'$-CS-Baer.
\end{enumerate}
\end{coll}

\begin{ex} \label{ec1} \rm (i) Consider the ring $R=\begin{pmatrix}\mathbb{Z}&\mathbb{Z}\\0&\mathbb{Z} \end{pmatrix}$
and the right $R$-modules $M_1=\begin{pmatrix}\mathbb{Z}&\mathbb{Z}\\0&0 \end{pmatrix}$ 
and $M_2=\begin{pmatrix}0&0\\0&\mathbb{Z} \end{pmatrix}$. Since
$M_1$ and $M_2$ are self-Baer right $R$-modules, $M_1$ and $M_2$ are self-CS-Baer. 
But we have seen in \cite[Example ~4.8]{CR1} that $R=M_1\oplus M_2$ is not a self-CS-Rickart right $R$-module, hence  $R$ is not a self-CS-Baer right $R$-module. 

(ii) The $\mathbb{Z}$-modules $\mathbb{Z}_2$ and $\mathbb{Z}_{16}$ are dual self-CS-Baer,
but the $\mathbb{Z}$-module $\mathbb{Z}_2\oplus \mathbb{Z}_{16}$ is not dual self-CS-Baer, because the $\mathbb{Z}$-module $\mathbb{Z}_2\oplus \mathbb{Z}_{16}$ is not dual self-CS-Rickart \cite[Example~4.8]{CR1}.
\end{ex}

\begin{theo} \label{t:dsum} 
Let $\mathcal{A}$ be an abelian category.
       \begin{enumerate}
                 \item Assume that $\mathcal{A}$ is AB3*. Then $\bigoplus\limits_{i=1}  
                         ^nN_i$ is (strongly) $M$-CS-Baer \, if and only if \,  $N_i$ is
                         (strongly) $M$-CS-Baer for any \, $i\in \{1,\dots ,n\}$.
                 \item Assume that $\mathcal{A}$ is AB3. Then $N$ is dual 
                          (strongly) $\bigoplus\limits_{i=1}^nM_i$-CS-Baer if and only if $N$  
                          is dual (strongly) $M_i$-CS-Baer for any $i\in \{1,\dots ,n\}$.
       \end{enumerate}
\end{theo}

\begin{proof} (1) By Proposition~\ref{csr-csb}, $\bigoplus\limits_{i=1}^nN_i$ is (strongly) $M$-CS-Baer if and only if $\bigoplus\limits_{i=1}^nN_i^I\simeq \left(\bigoplus\limits_{i=1}^nN_i\right)^I$ is (strongly) $M$-CS-Rickart for every set $I$. By \cite[Theorem~4.2]{CR1} (\cite[Theorem~4.2]{CR3}) we deduce that $N_i^I$ is (strongly) $M$-CS-Rickart for every $i\in \{1,\dots ,n\}$, and every set $I$. Then $N_i$ is (strongly) $M$-CS-Baer for every $i\in \{1,\dots ,n\}$, again by  Proposition~\ref{csr-csb}.
\end{proof}

Following \cite{DHSW,TV}, for any right $R$-module $M$, denote by $Z_2(M)$ the submodule of $M$ defined by $Z_2(M)/Z(M)=Z(M/Z(M))$ and $\overline{Z}^2(M)=\overline{Z}(\overline{Z}(M))$. Now we may give the following application of Theorem \ref{t:dsum}.

\begin{coll} Let $R$ be a unitary ring, and let $M$ be a right $R$-module. Then:
\begin{enumerate}
\item $Z_2(M)$ and $M/Z_2(M)$ are (strongly) $M$-CS-Baer if and only if $Z_2(M)$ is a direct summand of $M$ and $M$ is (strongly) self-CS-Baer.
\item $M$ is dual (strongly) $\overline{Z}^2(M)$-CS-Baer and dual (strongly) $M/\overline{Z}^2(M)$-CS-Baer if and only if $\overline{Z}^2(M)$ is a direct summand of $M$ and $M$ is dual (strongly) self-CS-Baer.
\end{enumerate}
\end{coll}

\begin{proof} (1) Assume that $Z_2(M)$ and $M/Z_2(M)$ are (strongly) $M$-CS-Baer. By the proof of \cite[Corollary~4.3]{CR1}, we have $M\cong Z_2(M)\oplus M/Z_2(M)$. Then $M$ is (strongly)-CS-Baer by Theorem \ref{t:dsum}. The converse holds by Corollary \ref{c:sdb5}.

(2) Assume that  $M$ is dual (strongly) $\overline{Z}^2(M)$-CS-Baer and dual (strongly) $M/\overline{Z}^2(M)$-CS-Baer. By the proof of \cite[Corollary~4.3]{CR1}, we have $M\cong \overline{Z}^2(M)\oplus M/\overline{Z}^2(M)$. Then $M$ is dual (strongly) self-CS-Baer by Theorem \ref{t:dsum}. The converse holds by Corollary \ref{c:sdb5}.
\end{proof}

\begin{prop} Let $(M_i)_{i\in I}$ be a family of objects of an abelian category $\mathcal{A}$.
        \begin{enumerate}
                 \item If $\mathcal{A}$ is AB3* and $\prod\limits_{i\in I}M_i$ is    
                         (strongly) self-CS-Baer, then $M_i$ is (strongly) $M_j$-CS-Rickart for every 
                         $i,j\in I$, and $M_i$ is (strongly) self-CS-Baer for every $i\in I$.
                \item If $\mathcal{A}$ is AB3 and $\bigoplus\limits_{i\in I}M_i$    
                         is dual (strongly) self-CS-Baer, then $M_i$ is dual (strongly) $M_j$-CS-Rickart  
                         for every $i,j\in I$, and $M_i$ is dual (strongly) self-CS-Baer for every $i\in I$.
       \end{enumerate}
\end{prop}

\begin{proof}
(1) Use \cite[Proposition~4.5]{CR1} (\cite[Proposition~4.4]{CR3}) and Corollary \ref{c:sdb5}.
\end{proof}

\begin{theo}

 Let $\mathcal{A}$ be an abelian category.
\begin{enumerate}
\item Assume that $\mathcal{A}$ is AB3*. Let $M$ be a (strictly) SSIP-extending object of $\mathcal{A}$, and let $(N_i)_{i\in I}$ be a family of objects of $\mathcal{A}$. Then $\prod\limits_{i\in I} N_i$ is (strongly) 
$M$-CS-Baer if and only if $N_i$ is (strongly) $M$-CS-Baer for every $i\in I$.
\item Assume that $\mathcal{A}$ is AB3. Let $(M_i)_{i\in I}$ be a family of objects of $\mathcal{A}$, and let $N$ be a (strictly) SSSP-lifting object of $\mathcal{A}$. Then $N$ is dual (strongly) $\bigoplus\limits_{i\in I} M_i$-CS-Baer if and only if $N$ is dual (strongly) $M_i$-CS-Baer for every $i\in I$.
\end{enumerate}
\end{theo}

\begin{proof} (1) By Proposition \ref{csr-csb}, $\prod\limits_{i\in I} N_i$ is (strongly) $M$-CS-Baer if and only if $\left(\prod\limits_{i\in I} N_i\right)^J\cong \prod\limits_{i\in I} N_i^J$ is (strongly) $M$-CS-Rickart for every set $J$. By \cite[Theorem~4.7]{CR1} (\cite[Theorem~4.7]{CR3}), this is equivalent to $N_i^J$ being (strongly) $M$-CS-Rickart for every $i\in I$ and every set $J$. But this means that $N_i$ is  (strongly) $M$-CS-Baer for every $i\in I$ by Proposition \ref{csr-csb}.
\end{proof}

Next we study the behaviour of the (strong) self-CS-Baer property with respect to direct sum decompositions.  

\begin{theo} \label{f-invariant-summands} Let $\mathcal{A}$ be an abelian category, 
and let $M=\bigoplus_{i\in I}M_i$ be a direct sum decomposition in $\mathcal{A}$ for some finite set $I$. Then:
\begin{enumerate}
\item 
\begin{enumerate}[(i)]
\item If $\Hom_{\mathcal{A}}(M_i,M_j)=0$ for every $i,j\in I$ with $i\neq j$, then $M$ is self-CS-Baer if and only if $M_i$ is self-CS-Baer for each $i\in I$.
\item $M$ is strongly self-CS-Baer if and only if $M_i$ is strongly self-CS-Baer for each $i\in I$ and $\Hom_{\mathcal{A}}(M_i,M_j)=0$ for every $i,j\in I$ with $i\neq j$.
\end{enumerate}
\item 
\begin{enumerate}[(i)]
\item If $\Hom_{\mathcal{A}}(M_i,M_j)=0$ for every $i,j\in I$ with $i\neq j$, then $M$ is  dual self-CS-Baer if and only if $M_i$ is dual self-CS-Baer for each $i\in I$.
\item $M$ is  dual strongly self-CS-Baer if and only if $M_i$ is dual strongly self-CS-Baer for each $i\in I$ and $\Hom_{\mathcal{A}}(M_i,M_j)=0$ for every $i,j\in I$ with $i\neq j$.
\end{enumerate}
\end{enumerate}
\end{theo}

\begin{proof} (2) Assume that $M$ is dual (strongly) self-CS-Baer. Then $M_i$ is dual (strongly) self-CS-Baer for every $i\in I$ by Corollary \ref{c:sdb5}. If $M$ is dual strongly self-CS-Baer, then $M$ is weak duo by Corollary \ref{st00}, hence each $M_i$ is fully invariant, and thus ${\rm Hom}_{\mathcal{A}}(M_i,M_j)=0$ for every $i,j\in I$ with $i\neq j$.

We prove the converse for $I=\{1,2\}$, the general case following inductively. Consider a direct sum decomposition $M=M_1\oplus M_2$ such that $M_1,M_2$ are dual (strongly) self-CS-Baer and $\Hom_{\mathcal{A}}(M_1,M_2)=\Hom_{\mathcal{A}}(M_2,M_1)=0$.
To prove that $M$ is dual (strongly) self-CS-Baer, let $(\varphi_{\lambda})_{\lambda\in\Lambda}$ be a family of endomorphisms of $M$. Note that $M_1$ and $M_2$ are fully invariant in $M$,
since $\Hom_{\mathcal{A}}(M_1, M_2)=\Hom_{\mathcal{A}}(M_2, M_1)=0$. Then, for every $\lambda \in \Lambda$ and each $i=1, 2$, we may consider the endomorphism $\varphi_{i\lambda}=\varphi_{\lambda} k_i: M_i \rightarrow M_i$ (note that ${\rm Im}(\varphi_{\lambda} k_i)\subseteq M_i$), where $k_i:M_i\to M$ is the inclusion morphism. 
Set $N_1=\sum_{\lambda \in \Lambda} {\rm Im}(\varphi_{1\lambda})$ and $N_2=\sum_{\lambda \in \Lambda}{\rm Im}(\varphi_{2\lambda})$.
Then $N=\sum_{\lambda \in \Lambda} {\rm Im}(\varphi_{\lambda})=N_1 \oplus N_2$. By assumption, there exist a (fully invariant) direct summand $K_1$ of $M_1$
and a (fully invariant) direct summand $K_2$ of $M_2$ such that $N_i/K_i$ is superfluous in $M_i/K_i$ for each $i=1, 2$. Set $K=K_1 \oplus K_2$, and let $X$ be a subobject of $M$ such that $K \subseteq X$ and $(N/K) + (X/K) = M/K$. Since $N_1/K_1$ is superfluous in $M_1/K_1$, we obtain $N_2+X=M$.
Therefore $X=M$, as $N_2/K_2$ is superfluous in $M_2/K_2$. Note that $K$ is a (fully invariant) direct summand of $M$. It follows that $M$ is a dual (strongly) self-CS-Baer object.
\end{proof}

In case of module categories, we may add a condition that allows us to deal with (possibly) infinite direct sum decompositions as follows.

\begin{theo} \label{t:pstr4} Let $R$ be a unitary ring, and let $M=\bigoplus_{i\in I}M_i$ be a direct sum decomposition of a right $R$-module $M$ into submodules $M_i$ such that for every submodule $L$ of $M$, $L=\bigoplus_{i\in I}(L\cap M_i)$. Then:
\begin{enumerate}
\item $M$ is (strongly) self-CS-Baer if and only if $M_i$ is (strongly) self-CS-Baer for each $i\in I$.
\item $M$ is  dual (strongly) self-CS-Baer if and only if $M_i$ is dual (strongly) self-CS-Baer for each $i\in I$.
\end{enumerate}
\end{theo}

\begin{proof} Let us first note that the condition that $L=\bigoplus_{i\in I}(L\cap M_i)$ for every submodule $L$ of $M$ implies that $\Hom_R(M_i,M_j)=0$ for every $i,j\in I$ with $i\neq j$ \cite[Lemma~2.4]{OHS}.

(2) We only need to prove the sufficiency, since the necessity follows by Corollary \ref{c:sdb5}. 
Assume that $M_i$ is dual (strongly) self-CS-Baer for each $i\in I$. Let $(f_j)_{j\in J}$ be a family of endomorphisms $f_j:M\to M$. For every $i\in I$, denote by $k_i:M_i\to M$ the inclusion morphism. Then ${\rm Im}(f_j)=\sum_{i\in I}{\rm Im}(f_jk_i)$ for each $j\in J$. Since ${\rm Im}(f_jk_i)\subseteq M_i$ for each $i\in I$ and $j\in J$, we have that ${\rm Im}(f_j)=\bigoplus_{i\in I} {\rm Im}(f_jk_i)$ for each $j\in J$. For each $i\in I$, consider the family $(f_jk_i)_{j\in J}$ of endomorphisms of $M_i$, and by hypothesis deduce the existence of a (fully invariant) direct summand $K_i$ of $M_i$ such that $N_i=\sum_{j\in J}{\rm Im}(f_jk_i)$ lies above $K_i$. Consider the direct summand $K=\bigoplus_{i\in I}K_i$ of $M$. If each $K_i$ is fully invariant in $M_i$, then $K$ is fully invariant in $M$ by \cite[Proposition~2.10]{CKT}. Denote $N=\sum_{j\in J}{\rm Im}(f_j)$. Then we have: \[N=\sum_{j\in J}\left(\bigoplus_{i\in I} {\rm Im}(f_jk_i)\right)=\bigoplus_{i\in I} \left(\sum_{j\in J}{\rm Im}(f_jk_i)\right)=\bigoplus_{i\in I} N_i.\] 
We claim that $N$ lies above $K$. To this end, let $X$ be a subobject of $M$ with $K\subseteq X$ such that $N/K+X/K=M/K$. Then we have:
\[M=\bigoplus_{i\in I} \left( N_i+(X\cap M_i) \right).\] Now fix $i\in I$. Then $M_i=N_i+(X\cap M_i)$. Since 
$N_i/K_i$ is superfluous in $M_i/K_i$, it follows that $M_i=K_i+(X\cap M_i)$. But $K_i\subseteq K\subseteq X$ and $K_i\subseteq M_i$, hence $K_i\subseteq X\cap M_i$. Thus $M_i=X\cap M_i$, whence we deduce that $M=\bigoplus_{i\in I}M_i=\bigoplus_{i\in I}(X\cap M_i)=X$. This shows that $N$ lies above $K$, and consequently, $M$ is dual (strongly) self-CS-Baer.
\end{proof}

 \section{Dual self-CS-Baer modules over Dedekind domains}

The aim of this section is to determine the structure of dual (strongly) self-CS-Baer modules over Dedekind domains.
The next theorem shows that we can reduce the problem to the case of modules over discrete valuation rings.
First, we present the following example.

\begin{ex} \rm \label{injective} Let $R$ be a Dedekind domain. It is well known that the class of injective $R$-modules is closed under
factor modules and direct sums. Let $M$ be an injective $R$-module. Then for every family $(f_{\lambda})_{\lambda\in \Lambda}$ of endomorphisms of $M$,
$\sum_{\lambda \in \Lambda} \Im (f_{\lambda})$ is a direct summand of $M$. Therefore, $M$ is a dual self-CS-Baer $R$-module.
\end{ex}

Let $M$ be a module over a Dedekind domain $R$. We denote by $T(M)$ the torsion submodule of $M$, i.e., $T(M)=\{x \in M \mid {\rm Ann}_R(x) \neq 0\}$.
Let $\mathbf{P}$ denote the set of non-zero prime ideals of $R$. For any $\mathfrak{p} \in \mathbf{P}$, the $\mathfrak{p}$-primary component of $M$
will be denoted by $T_{\mathfrak{p}}(M)$, that is, $T_{\mathfrak{p}}(M) = \{x \in M \mid \mathfrak{p}^nx=0$ for some integer $n \geq 0\}$.

\begin{rem} \label{sum-summands} \rm (i) Let $M$ be a dual self-CS-Baer module. Note that every direct summand of $M$ is a homomorphic image of $M$. Then for every family of direct summands $(M_{\lambda})_{\lambda\in\Lambda}$ of $M$, $\sum_{\lambda \in \Lambda} M_{\lambda}$ lies above a direct summand of $M$, hence $M$ is LSSSP (also, see Corollary \ref{c:EL}). 

(ii) Let $\mathfrak{p} \in \mathbf{P}$ and let $M$ be a $\mathfrak{p}$-primary $R$-module (i.e., $T_{\mathfrak{p}}(M)=M$).
Then $M$ is clearly a module over the localization $R_{\mathfrak{p}}$. Moreover, $M$ is a dual self-CS-Baer $R$-module if and only if $M$ is a dual self-CS-Baer $R_{\mathfrak{p}}$-module.
\end{rem}

\begin{theo} \label{structure-Dedekind} Let $R$ be a non-local Dedekind domain with quotient field $K$, and let $M$ be an $R$-module. Then the following are equivalent:
\begin{enumerate}
\item[{\em (i)}] $M$ is a dual self-CS-Baer module;
\item[{\em (ii)}]  $M=T(M) \oplus L$ such that $T(M)$ is a dual self-CS-Baer module and $L \cong K^{(I)}$ for some index set $I$;
\item[{\em (iii)}]  $M=\left(\bigoplus_{\mathfrak{p} \in \mathbf{P}} T_{\mathfrak{p}}(M)\right) \oplus L$ such that each $T_p(M)$ $(\mathfrak{p} \in \mathbf{P})$
is a dual self-CS-Baer $R_{\mathfrak{p}}$-module and $L \cong K^{(I)}$ for some index set $I$.
\end{enumerate}
\end{theo}

\begin{proof} (i) $\Rightarrow$ (ii) It is well known that $T(M)=\bigoplus_{\mathfrak{p} \in \mathbf{P}} T_{\mathfrak{p}}(M)$. By \cite[Theorem 4.11]{Tribak}, each $T_{\mathfrak{p}}(M)$ is a direct summand of $M$. Hence $T(M)$ lies above a direct summand of $M$ by Remark \ref{sum-summands}(i).
It follows that $T(M)=A \oplus B$ such that $A$ is a direct summand of $M$
and $B \ll M$ (i.e., $B$ is superfluous in $M$). Note that $A=T(A)=\bigoplus_{\mathfrak{p} \in \mathbf{P}} T_{\mathfrak{p}}(A)$ and $B=T(B)=\bigoplus_{\mathfrak{p} \in \mathbf{P}} T_{\mathfrak{p}}(B)$. Then $T(M)=\bigoplus_{\mathfrak{p} \in \mathbf{P}} (T_{\mathfrak{p}}(A) \oplus T_{\mathfrak{p}}(B))$, and so $T_{\mathfrak{p}}(M)=T_{\mathfrak{p}}(A) \oplus T_{\mathfrak{p}}(B)$ for all $\mathfrak{p} \in \mathbf{P}$. Since each $T_{\mathfrak{p}}(B)$ is superfluous in $M$, we have $T_{\mathfrak{p}}(B) \ll T_{\mathfrak{p}}(M)$ for all $\mathfrak{p} \in \mathbf{P}$ by \cite[Lemma 4.2(2)]{MM}. Therefore $T_{\mathfrak{p}}(B)=0$ for all $\mathfrak{p} \in \mathbf{P}$. This implies that $B=0$, and hence $T(M)=A$ is a direct summand of $M$. Let $L$ be a submodule of $M$ such that $M=T(M) \oplus L$. By \cite[Proposition 4.13]{Tribak}, $L \cong K^{(I)}$ for some index set $I$.

(ii) $\Rightarrow$ (i) Let $(f_{\lambda})_{\lambda\in\Lambda}$ be a family of endomorphisms of $M$. Set $N=\sum_{\lambda \in \Lambda} \Im (f_{\lambda})$.
Then $N=D + A$, where $D=\sum_{\lambda \in \Lambda} f_{\lambda}(L)$ and $A=\sum_{\lambda \in \Lambda} f_{\lambda}(T(M))$.
Since $R$ is a Dedekind domain, $D$ is injective. Hence $D$ is a direct summand of $M$.
Thus $M=D \oplus B$ for some submodule $B$ of $M$. By modularity, we obtain $N=D \oplus (N \cap B)$.
Note that $N \cap B \cong N/D \cong A/(D \cap A)$. Therefore $N \cap B \subseteq T(M)$ as $A \subseteq T(M)$.
Let $\theta:A/(D \cap A) \rightarrow N \cap B$ be an isomorphism, and let $\pi:A \rightarrow A/(D \cap A)$ be the projection map.
Set $\alpha=\mu\theta\pi:A \rightarrow T(M)$, where $\mu:N \cap B \rightarrow T(M)$ is the inclusion map.
For each $\lambda \in \Lambda$, let $g_{\lambda}:T(M) \rightarrow T(M)$ be the endomorphism of $T(M)$ defined by $g_{\lambda}(x)=\alpha f_{\lambda}(x)$ for all $x \in T(M)$. It follows that $\sum_{\lambda \in \Lambda} \Im g_{\lambda} = \alpha(A) = N \cap B$. Since $T(M)$ is dual self-CS-Baer,
$N \cap B=U \oplus V$ such that $U$ is a direct summand of $T(M)$ and $V \ll T(M)$. This implies that $N=D \oplus U \oplus V$ and $V \ll M$.
Moreover, it is clear that $D \oplus U$ is a direct summand of $M$, since $U$ is a direct summand of $B$. Consequently, $M$ is dual self-CS-Baer.

(ii) $\Leftrightarrow$ (iii) This follows from Theorem \ref{t:pstr4} and Remark \ref{sum-summands}(ii), noting that for every submodule $N$ of $T(M)$, $N=\bigoplus_{\mathfrak{p} \in \mathbf{P}} T_{\mathfrak{p}}(N)=\bigoplus_{\mathfrak{p} \in \mathbf{P}} (N\cap T_{\mathfrak{p}}(M))$.
\end{proof}

Now our purpose is to describe the structure of dual self-CS-Baer modules over discrete valuation rings. In the remainder of this section we assume that $R$ is a discrete valuation ring with maximal ideal $\mathfrak{m}$, quotient field $K$ and $Q=K/R$. We begin with a characterization of dual self-CS-Baer modules having a direct summand isomorphic to $R$.

\begin{prop} \label{lifting} Let $R$ be a ring which is not necessarily a Dedekind domain, and let $M$ be an $R$-module.
Then the following statements are equivalent:
\begin{enumerate}
\item[{\em (i)}] $M \oplus R$ is dual self-CS-Baer;

\item[{\em (ii)}] $M \oplus R$ is a lifting module.
\end{enumerate}
\end{prop}

\begin{proof} (i) $\Rightarrow$ (ii) Let $a \in M \oplus R$ and consider the $R$-homomorphism $\varphi_a: M \oplus R \rightarrow M \oplus R$
defined by $\varphi_a((x, r))=ar$ for all $x \in M$ and all $r \in R$. Let $L$ be a submodule of $M \oplus R$.
It is easily seen that $L=\sum_{a \in L}\varphi_a(M \oplus R)$. Since $M \oplus R$ is dual self-CS-Baer, it follows that $L$ lies above a direct summand. Therefore $M \oplus R$ is a lifting module.

(ii) $\Rightarrow$ (i) This is evident.
\end{proof}

An internal direct sum $\bigoplus_{i \in I} A_i$ of submodules of a module $M$ is called a {\it local direct summand of $M$} if $\bigoplus_{i \in F} A_i$ is a direct summand of $M$ for any finite subset $F \subseteq I$. If, moreover, $\bigoplus_{i \in I} A_i$ is a direct summand of $M$, then we say that
the local direct summand $\bigoplus_{i \in I} A_i$ is also a direct summand of $M$.
Recall that a family of modules $(M_{\alpha})_{\alpha \in \Lambda}$ is called {\it locally-semi-T-nilpotent} if for any subfamily $(M_{\alpha_i})_{i \in \mathbb{N}}$ with distinct $\alpha_i$ and any family of non-isomorphisms $(f_i: M_{\alpha_i} \rightarrow M_{\alpha_{i+1}})_{i \in \mathbb{N}}$,
and for every $x \in M_{\alpha_1}$, there exists $n \in \mathbb{N}$ (depending on $x$) such that $f_n \ldots f_2f_1(x)=0$.
For natural numbers $n_1$, $n_2$, $\ldots$, $n_s$, let $B(n_1, n_2, \ldots, n_s)$ denote the direct sum of arbitrarily many copies of $R/\mathfrak{m}^{n_1}$, $R/\mathfrak{m}^{n_2}$, $\ldots$, $R/\mathfrak{m}^{n_s}$.

The proof of the next lemma follows similar arguments as in \cite[Theorem 3.13]{Tribak}, but we include it for the convenience of the reader.

\begin{lemm} \label{main-lemma} Let $M=D_1 \oplus D_2 \oplus L$ such that $D_1 \cong K^{(I_1)}$, $D_2 \cong Q^{(I_2)}$ and $L \cong B(n)$
for some index sets $I_1$ and $I_2$ and some positive integer $n$. Let $N$ be a submodule of $M$ having the form $N=D+X$, where $D$ is an injective
submodule of $M$ and $X$ is a submodule of $M$ with $X\mathfrak{m}^n=0$. Then $N$ lies above a direct summand of $M$.
\end{lemm}

\begin{proof} Since $D$ is a direct summand of $N$, there exists a submodule $Y$ of $N$ such that $N=D \oplus Y$.
Note that $Y \cong X/(X \cap D)$. Thus $Y\mathfrak{m}^n=0$. Using \cite[Theorem 6.14]{Sharpe}, $Y=\bigoplus_{i \in I} Y_i$ is a direct sum of cyclic submodules
$Y_i=y_iR$ $(i \in I)$. By \cite[Lemma 3.9]{Tribak}, we infer that each $Y_i$ is either superfluous in $M$ or a direct summand of $M$.
Let $J=\{i \in I \mid Y_i \ll M\}$, and put $U=\bigoplus_{j \in J}Y_j$. Then $U \ll M$ by \cite[Lemma 3.12]{Tribak}.
If $J=I$, then $Y=U \ll M$ and we are done. Now assume that $I \setminus J \neq \emptyset$ and set $V=\bigoplus_{i \in I \setminus J}Y_j$.
Hence $Y=U \oplus V$. Note that $M$ is a direct sum of copies of $K$, $Q$ and $R/{\mathfrak{m}}^n$. Then this decomposition complements maximal direct summands
by \cite[Theorem 12.6]{AF}. This implies that $Y_i \cong R/{\mathfrak{m}}^n$ for every $i \in I \setminus J$.
By Zorn's Lemma, there exists a maximal subset $\Lambda \subseteq I \setminus J$ such that $W=\bigoplus_{\lambda \in \Lambda} Y_{\lambda}$
is a local direct summand of $M$. In addition, $(Y_{\lambda})_{\lambda \in \Lambda}$ is a locally-semi-T-nilpotent family of indecomposable
direct summands by \cite[Lemma 3.11]{Tribak}. Thus $W$ must be a direct summand of $M$ by \cite[Theorem 3.2]{Dung}.
Therefore  $M=W \oplus E$ for some submodule $E$ of $M$. Let $\Omega=(I \setminus J)\setminus \Lambda$.
Then $V=W \oplus \left(\bigoplus_{\alpha \in \Omega}y_{\alpha}R\right)$. Fix $\alpha \in \Omega$. Then there exist $w_{\alpha} \in W$ and $e_{\alpha} \in E$
such that $y_{\alpha}=w_{\alpha} + e_{\alpha}$. Moreover, we have $W\oplus y_{\alpha}R = W\oplus e_{\alpha}R$.
Now maximality of $\Lambda$ implies that $e_{\alpha}R$ is not a direct summand of $M$. Since $e_{\alpha}{\mathfrak{m}}^n=0$, we have $e_{\alpha}R \ll M$
by \cite[Lemma 3.9]{Tribak}. Hence $\sum_{\alpha \in \Omega} e_{\alpha}R \ll M$ by \cite[Lemma 3.12]{Tribak}.
Note that $V=W \oplus \left(\sum_{\alpha \in \Omega}e_{\alpha}R\right)$.
Then $$N=D \oplus Y = D \oplus U \oplus V = D \oplus U \oplus W \oplus \left(\sum_{\alpha \in \Omega}e_{\alpha}R\right).$$
As $D$ is injective, we have $D \subseteq d(W) \oplus d(E)$, where $d(Z)$ denotes the largest divisible submodule of a module $Z$. Since $W{\mathfrak{m}}^n=0$, we have $d(W)=0$.
Therefore $D \subseteq d(E) \subseteq E$. This implies that $D \oplus W$ is a direct summand of $M$.
Further, we have $U \oplus (\oplus_{\alpha \in \Omega}e_{\alpha}R) \ll M$. It follows that $N$ lies above a direct summand of $M$.
\end{proof}

\begin{prop} \label{D1-D2-B(n)} Let $M=D_1 \oplus D_2 \oplus L$ be such that $D_1 \cong K^{(I_1)}$, $D_2 \cong Q^{(I_2)}$ and $L \cong B(n)$ for some positive integer $n$. Then $M$ is a dual self-CS-Baer module.
\end{prop}

\begin{proof} Let $(\varphi_{\lambda})_{\lambda\in\Lambda}$ be a family of endomorphisms of $M$. For each $\lambda \in \Lambda$,
$\Im (\varphi_{\lambda})=D_{\lambda}+L_{\lambda}$, where $D_{\lambda}=\varphi_{\lambda}(D_1 \oplus D_2)$ and $N_{\lambda}=\varphi_{\lambda}(L)$.
Therefore $\sum_{\lambda \in \Lambda} \Im (\varphi_{\lambda}) = D+N$, where $D=\sum_{\lambda \in \Lambda}D_{\lambda}$
and $N=\sum_{\lambda \in \Lambda}N_{\lambda}$. Note that $D$ is injective as $R$ is a Dedekind domain.
Moreover, we have $N_{\lambda}{\mathfrak{m}}^n=0$ for all $\lambda \in \Lambda$.
Hence $N{\mathfrak{m}}^n=0$. From Lemma \ref{main-lemma}, we infer that $\sum_{\lambda \in \Lambda} \Im (\varphi_{\lambda})$ lies above a direct summand of $M$. It follows that $M$ is dual self-CS-Baer.
\end{proof}

\begin{theo} \label{structure-DVR} Let $R$ be a discrete valuation ring with maximal ideal $\mathfrak{m}$, quotient field $K$ and $Q=K/R$. Let $I_1$ and $I_2$ be two index sets and let $a$, $b$, $c$ and $n$ be non-negative integers. Then an $R$-module $M$ is dual self-CS-Baer if and only if $M$ is isomorphic to one of the following modules:

{\em (1)} $K^a \oplus Q^b \oplus R^c$ with $a \leq 1$ if $R$ is incomplete, or

{\em (2)} $K^{(I_1)} \oplus Q^{(I_2)} \oplus B(n)$, or

{\em (3)} $K^{(I_1)} \oplus B(n, n+1)$.
\end{theo}

\begin{proof} Using the fact that every dual self-CS-Baer module is dual self-CS-Rickart, it suffices to characterize when the modules given in \cite[Theorem 3.14]{Tribak} are dual self-CS-Baer. Let $M$ be an $R$-module.

{\bf Case 1:} Suppose that $M \cong K^{({\Lambda_1})} \oplus Q^{({\Lambda_2})} \oplus R^n$, where ${\Lambda_1}$ and ${\Lambda_2}$ are index sets
and $n$ is a non-negative integer. By Example \ref{injective}, $K^{({\Lambda_1})} \oplus Q^{({\Lambda_2})}$ is dual CS-Baer.
Now assume that $n \neq 0$. Then $M$ is dual self-CS-Baer if and only if $M$ is lifting by Proposition \ref{lifting}.
This is equivalent to $M \cong K^{a} \oplus Q^{b} \oplus R^n$ with $a \leq 1$ if $R$ is incomplete by \cite[Proposition A.7]{MM}.

{\bf Case 2:} Suppose that $M \cong K^{({\Lambda_1})} \oplus Q^{({\Lambda_2})} \oplus B(n)$, where ${\Lambda_1}$ and ${\Lambda_2}$
are index sets and $n$ is a non-negative integer. Then $M$ is dual self-CS-Baer by Proposition \ref{D1-D2-B(n)}.

{\bf Case 3:} Suppose that $M = D \oplus L$, where $D \cong K^{({\Lambda})}$ and $L \cong B(n, n+1)$, where ${\Lambda}$ is an index set
and $n$ is a non-negative integer. Note that $\Hom_R(D, L)=\Hom_R(L, D)=0$. By \cite[Proposition A.7]{MM}, $B(n, n+1)$ is a lifting module.
Moreover, $D$ is dual self-CS-Baer, since $R$ is a Dedekind domain (see Example \ref{injective}). Hence $M$ is dual self-CS-Baer by Theorem \ref{f-invariant-summands}. This completes the proof.
\end{proof}

Let $R$ be a Dedekind domain with quotient field $K$, and let $\mathfrak{p}$ be a non-zero prime ideal of $R$. By $B_{\mathfrak{p}}(n_1, \ldots, n_s)$
we denote the direct sum of arbitrarily many copies of $R/\mathfrak{p}^{n_1}$, $R/\mathfrak{p}^{n_2}$, $\ldots$, $R/\mathfrak{p}^{n_s}$ for some non-negative integers $n_1$, $n_2$, $\ldots$, $n_s$. We will denote by $R(\mathfrak{p}^{\infty})$ the $\mathfrak{p}$-primary component of the torsion $R$-module $K/R$.
Combining Theorems \ref{structure-Dedekind} and \ref{structure-DVR}, we obtain the following structure result.

\begin{theo} \label{structure-Dedekind-2} Let $R$ be a non-local Dedekind domain with quotient field $K$ and let $M$ be an $R$-module.
Then the following are equivalent:
\begin{enumerate}
\item[{\em (i)}] $M$ is a dual self-CS-Baer module;

\item[{\em (ii)}] $M=(\bigoplus_{\mathfrak{p} \in \mathbf{P}} T_{\mathfrak{p}}(M)) \oplus L$ such that $L \cong K^{({\Lambda})}$ for some index set ${\Lambda}$ and for every non-zero prime ideal $\mathfrak{p}$
of $R$, $T_{\mathfrak{p}}(M) \cong R(\mathfrak{p}^{\infty})^{(\Lambda)} \oplus B_{\mathfrak{p}}(n)$ or $T_{\mathfrak{p}}(M) \cong B_{\mathfrak{p}}(n, n+1)$, where ${\Lambda}$ is an index set and $n$ is a non-negative integer.
\end{enumerate}
\end{theo}

Next, we exhibit some examples of dual self-CS-Rickart modules which are not dual self-CS-Baer.

\begin{ex} \label{e:dual} \rm (i) Consider the $\mathbb{Z}$-module $M = \prod_{p\; {\rm prime}} \mathbb{Z}_p$. Since $\End_{\mathbb{Z}}(M)$ is isomorphic to the ring $\prod_{p\; {\rm prime}} \mathbb{Z}_p$, $\End_{\mathbb{Z}}(M)$ is a von Neumann regular ring. Hence $M$ is a dual self-CS-Rickart module (see \cite[Example 4.5]{Tribak}). On the other hand, it is easily seen that $T(M)=\bigoplus_{p\; {\rm prime}} \mathbb{Z}_p$ is not a direct summand of $M$. Therefore, $M$ is not dual self-CS-Baer by Theorem \ref{structure-Dedekind}.

(ii) Let $R$ be a discrete valuation ring with quotient field $K$. Comparing \cite[Theorem 3.14]{Tribak} with Theorem \ref{structure-DVR}, we obtain many examples of dual self-CS-Rickart $R$-modules which are not dual self-CS-Baer. For example, for any positive integer $n$, the $R$-modules
$K^{(\mathbb{N})} \oplus R^n$, $(K/R)^{(\mathbb{N})} \oplus R^n$ and $K^{(\mathbb{N})} \oplus (K/R)^{(\mathbb{N})} \oplus R^n$ are dual self-CS-Rickart, but they are not dual self-CS-Baer.
\end{ex}

\begin{rem} \rm Let $M$ be a dual strongly self-CS-Baer module. Let $N$ be a direct summand of $M$ and let $\pi:M \rightarrow N$ be the projection map.
Then $\Im (\pi) = N$ lies above a fully invariant direct summand of $M$. Therefore $N$ is fully invariant in $M$.
This implies that $M$ cannot have a direct summand isomorphic to $N \oplus N$.
\end{rem}

Combining Theorems \ref{f-invariant-summands} and \ref{t:pstr4}, Theorems \ref{structure-DVR} and \ref{structure-Dedekind-2},  and the preceding remark, we get the following corollary.

\begin{coll} \label{strong-d-Baer} Let $R$ be a Dedekind domain with quotient field $K$ and let $M$ be an $R$-module.
\begin{enumerate}
\item[{\em (i)}] If $R$ is a discrete valuation ring, then $M$ is dual strongly self-CS-Baer if and only if $M \cong R$ or $M \cong K$ or $M \cong K/R$ or $M \cong R/\mathfrak{m}^n$ or $M \cong K \oplus R/\mathfrak{m}^n$ for some non-negative integer $n$.
\item[{\em (ii)}] If $R$ is not local, then $M$ is dual strongly self-CS-Baer if and only if one of the following conditions holds:
\begin{enumerate}
\item[{\em (a)}] $M=\left(\bigoplus_{\mathfrak{p} \in \mathbf{P}} T_{\mathfrak{p}}(M)\right) \oplus L$ with $L \cong K$ and for every non-zero prime ideal $\mathfrak{p}$ of $R$, there exists a non-negative integer $n_p$ depending on $\mathfrak{p}$ such that $T_{\mathfrak{p}}(M) \cong R/{\mathfrak{p}}^{n_p}$.

\item[{\em (b)}] $M=\bigoplus_{\mathfrak{p} \in \mathbf{P}} T_{\mathfrak{p}}(M)$ such that for every non-zero prime ideal $\mathfrak{p}$ of $R$,
$T_{\mathfrak{p}}(M) \cong R(\mathfrak{p}^{\infty})$ or $T_{\mathfrak{p}}(M) \cong R/{\mathfrak{p}}^{n_p}$ for some non-negative integer $n_p$ depending on $\mathfrak{p}$.
 \end{enumerate}
\end{enumerate}
\end{coll}
 
\begin{ex} \rm Let $R$ be a Dedekind domain with quotient field $K$. Comparing Theorem \ref{structure-Dedekind-2} and Corollary \ref{strong-d-Baer}, we see that $K^{(\mathbb{N})}$ is dual self-CS-Baer, but not dual strongly self-CS-Baer.
\end{ex}

\end{document}